\documentclass[a4paper,11pt]{article}
\usepackage{graphicx}
\usepackage{hyperref}
\usepackage[T1]{fontenc}
\usepackage[utf8]{inputenc}
\usepackage[english]{babel}
\usepackage{amsthm,amssymb,amsfonts}
\usepackage{amsmath}
\usepackage{mathtools}
\usepackage{mathrsfs}
\usepackage{lmodern}
\usepackage{subcaption}
\usepackage{geometry}
\usepackage[boxed]{algorithm2e}
\usepackage{cite}

\usepackage{framed}
\usepackage{xcolor}
\definecolor{darkgreen}{rgb}{0.00,0.5,0.00}
\usepackage{hyperref}
\hypersetup{
    colorlinks=true,
    linkcolor=blue,
    citecolor=darkgreen,
    urlcolor=blue
}

\usepackage{lscape}
\usepackage{enumitem}
\usepackage{array}

\newtheorem{theorem}{Theorem}[section]
\theoremstyle{definition}
\newtheorem{lemma}[theorem]{Lemma}
\newtheorem{proposition}[theorem]{Proposition}
\newtheorem{remark}[theorem]{Remark}
\newtheorem{corollary}[theorem]{Corollary}
\newtheorem{definition}[theorem]{Definition}
\newtheorem{example}[theorem]{Example}

\newcommand{\R}{\mathbb R}

\newcommand{\E}{\mathbb E}
\newcommand{\N}{\mathbb N}

\newcommand{\setto}{\rightrightarrows}

\DeclareMathOperator{\proj}{proj}
\DeclareMathOperator{\epi}{epi}
\DeclareMathOperator{\dist}{dist}
\DeclareMathOperator{\dom}{dom}

\DeclareMathOperator{\conv}{co}

\DeclareMathOperator*{\argmin}{argmin}
\DeclareMathOperator*{\argmax}{argmax}

\usepackage{accents}

\newcommand{\rev}[1]{{#1}}

\newcommand{\Pp}{\mathbb{P}}
\newcommand{\B}{\mathbb{B}}
\newcommand{\Rex}{\overline{\R}}

\newcommand{\emp}{\emptyset}
\newcommand{\Intf}[1]{\mathtt{E}_{#1}}
\DeclareMathOperator{\cl}{cl}
\DeclareMathOperator{\inter}{int}
\DeclareMathOperator{\bd}{bd}
\DeclareMathOperator{\gph}{gph}

\title{Lower- and upper-$C^2$ integrands: a structured class of functions for stochastic nonconvex optimization}
\author{%
Felipe Atenas$^{1}$ \quad Alejandro Jofr\'e$^{1,2,*}$ \quad
Pedro P\'erez-Aros$^{1,2}$  \quad David Torregrosa-Bel\'en$^{3}$\\[0.6em]
\small $^{1}$Centro de Modelamiento Matem\'atico (CNRS IRL 2807), Universidad de Chile, Santiago, Chile\\
\small $^{2}$Departamento de Ingenier\'ia Matem\'atica, Universidad de Chile, Santiago, Chile\\
\small $^{3}$ Department of Mathematics, University of Alicante, Alicante, Spain\\
[0.4em]
\small \texttt{fatenas@cmm.uchile.cl}; \texttt{ajofre@uchile.cl};
\small \texttt{pperez@dim.uchile.cl};
\small \texttt{david.torregrosa@ua.es}\\
\small $^{*}$Corresponding author
}
\date{}

\begin{document}
\maketitle

\begin{abstract}

We study nonconvex stochastic functions that admit a structured second-order representation, a generalization of lower- and upper-$C^2$ functions to normal integrands. Similarly to the deterministic case, lower-$C^2$ normal integrands are defined as the pointwise maxima over a compact random index set of functions that are twice-continuously differentiable in the decision variable with an integrable bounded Hessian. Our main result is a local value function characterization of this class of functions, namely, that they satisfy a local sufficient ascent property almost surely. Equivalently, they differ from a convex normal integrand by a random quadratic.
The upper-$C^2$ theory follows by a change of sign, and delivers the corresponding lower envelopes, the form in which these integrands appear as second-stage value functions of stochastic programs with compact recourse. We also show that for this class integrability of the function values, of the local Lipschitz moduli and of the minimum-norm Clarke subgradients are equivalent. We then study smooth integrands, finite lower envelopes, difference-of-convex integrands, Moreau and forward--backward envelopes, compositions with a smooth map, monotone convex compositions, and compact-recourse stochastic programs within these classes of nonconvex stochastic functions. As an application, we show that the optimistic loss of partial-label regression and the variational formulation of the conditional value-at-risk fit within this framework.
\end{abstract}



\section{Introduction}\label{sec:introduction}

Many stochastic optimization models \rev{involve} expected-value functionals of the form
\begin{equation}\label{eq:expected_objective_intro}
	x\mapsto \Intf{\varphi}(x):=\int_\Xi \varphi(\xi,x)\,\Pp(d\xi),
\end{equation}
where $(\Xi,\mathcal F,\Pp)$ is a probability space and $\varphi$ is a normal integrand. Such functionals play a central role in stochastic programming, variational analysis, and sample-average approximation; see, e.g., \cite{MR4362585,rockafellarwets1998,MR4436019}. A persistent difficulty is that the sample functions $\varphi_\xi:=\varphi(\xi,\cdot)$ may fail to be smooth even when the underlying model data are smooth. This phenomenon arises naturally in two-stage and recourse models: the first-stage objective is defined through the optimal value of the second-stage problem, and such value functions are typically nondifferentiable in the first-stage variable.

The purpose of this work is to formulate and study  structured classes of nonconvex functions at the level of normal integrands. In particular, we focus on stochastic versions of lower- and upper-$C^2$ functions (see, e.g., \cite[Definition 10.29]{rockafellarwets1998}) which enjoy a favorable  second-order structure.   More precisely,  a locally Lipschitz function is \rev{lower-$C^2$} if, locally, it can be represented as a maximum of $C^2$ functions. 
The \rev{upper-$C^2$} property is its natural counterpart, obtained by changing sign. 
The classes of lower- and upper-$C^2$ functions play a prominent role in nonsmooth and variational analysis: convex functions provide basic examples of lower-$C^2$ functions, whereas lower envelopes and value functions obtained through minimization typically lead to upper-$C^2$ behavior; we refer the reader to \cite{clarke1983optimization,MR1363364,rockafellarwets1998,ThibaultBook1,ThibaultBook2} for more details. Beyond their theoretical significance, these classes are also relevant to practical optimization problems and to the design and convergence analysis of numerical algorithms. The benign form of nonconvexity of the family of lower-$C^2$ functions, also known as locally \emph{weakly convex} functions, has been exploited to design suitable proximal algorithms (see, e.g., \cite{MR650442,MR2583885,MR2678400,MR3961352,atenas2023unified}). In contrast, the upper-$C^2$ structure has received comparatively less attention in the development of new algorithms (see, e.g., \cite{aragon2025coderivative,aragonartacho2025nonmonotonesubgradientmethodsbased,aragonartacho2023boosted}). Related developments also include Frank–Wolfe methods for the broader class of upper-$C^{1,\alpha}$ functions \cite{deoliveira2023note}.

We define the lower- and upper-$C^2$ classes of normal integrands as the functions that can be obtained by optimizing, over a compact random index set, an objective that is twice-continuously differentiable with respect to the decision variable and whose Hessian is dominated by an integrable random modulus (see Definition~\ref{def:upperC2normal}). In the lower-$C^2$ case, this holds via maximization, and, after the change of sign, via minimization in the upper-$C^2$ case. This formulation is useful in applications because it reduces the verification of lower- or upper-$C^2$ regularity to standard hypotheses: compactness of the index (recourse) set, Carath\'eodory regularity, and an integrable second-order bound. In this way, the result separates the geometric content of the lower- and upper-$C^2$ property from the measurability issues handled by the theory of normal integrands and measurable multifunctions.

Our main contribution is a local ``ascent lemma'' (resp.   ``descent lemma'') for lower-$C^2$ (resp. upper-$C^2$) normal integrands. Moreover, we also prove that a normal integrand $\varphi_\xi$ is lower-$C^2$ around a reference point with integrable curvature if and only if  
it is locally characterized by a stochastic upper envelope of quadratic-affine functions with an integrable quadratic coefficient.  The mirror result states that  
  lower envelopes of stochastic quadratic-affine functions are the upper-$C^2$ normal integrands. In particular, the latter have the form of the integrands that \rev{appear} as second-stage value functions in stochastic programming.

In the stochastic setting, the \rev{quadratic}  local  curvature modulus   of lower- and upper-$C^2$ functions  may \rev{naturally} depend on the random parameter $\xi$, \rev{an} essential dependence in applications. We therefore require \rev{it} to belong to an appropriate space $L^p_\Pp$.  We show that  this integrability requirement is mild enough to cover random value function constructions, but strong enough to transfer samplewise estimates to the expected functional \eqref{eq:expected_objective_intro}. In particular, it provides a convenient bridge between the local geometry of the sample functions and the analytic properties of the expected objective.

The paper also records how the integrability hypotheses interact with one another. For lower- and upper-$C^2$ normal integrands we show that $p$-integrability of the values on a neighborhood of a point, $p$-integrability of the local Lipschitz modulus, and $p$-integrability of the minimum-norm Clarke subgradient are equivalent, and that each of them needs to be checked at a single point only. We further relate these conditions to integrability assumptions commonly used in smooth stochastic optimization, such as bounded-variance and multiplicative-noise estimates.

The paper is organized as follows. Section~\ref{sec:preliminaries} fixes notation and recalls the required tools from normal-integrand theory and variational analysis. Section~\ref{sec:main_results}  presents equivalent characterizations of lower- and upper-$C^2$ normal integrands. Section~\ref{sec:examples_classes} collects important general classes of such integrands, including difference-of-convex functions, Moreau and forward--backward envelopes, and composite functions. Section~\ref{sec:lipschitz_moduli} presents the relationship between integrability of the values, of the Lipschitz moduli, and of the minimum-norm Clarke subgradients of lower- and upper-$C^2$ normal integrands. In Section~\ref{sec:app}, we describe two example applications that belong to the families of lower- and upper-$C^2$ normal integrands. 

\section{Notation and preliminaries}\label{sec:preliminaries}

Throughout the paper, $(\Xi,\mathcal F,\Pp)$ denotes a complete probability space. All assertions involving the parameter $\xi$ are understood to hold for $\Pp$-almost every $\xi\in\Xi$. After modifying the corresponding objects on a null set if necessary, we shall simply write ``for all $\xi\in\Xi$''. The space $\R^d$ is endowed with the Euclidean inner product $\langle\cdot,\cdot\rangle$ and the induced norm $\|\cdot\|$. We denote by $\B_r(x)$ and $\B_r[x]$ the open and closed balls of radius $r>0$ centered at $x$, respectively, and set $\B:=\B_1[0]$. We denote the extended real line by $\Rex:=[-\infty,+\infty]$. 

We say that a function $f:\R^d\to\R$ is $C^1$ if it is \emph{continuously differentiable}, that is, differentiable with a continuous gradient, and that it is $C^2$ if it is twice-continuously differentiable. 

For $A\subset\R^d$, the symbols $\cl A$, $\inter A$, $\bd A$, $\conv A$, and $\dist(x,A)$ denote, respectively, the closure, interior, boundary, convex hull, and distance from $x$ to $A$. For an extended-real-valued function $f:\R^d\to\Rex$, we denote the domain of $f$ by $\dom f:=\{x\in\R^d:f(x)<+\infty\}$ and the epigraph of $f$ by $\epi f:=\{(x,\alpha)\in\R^d\times\R:f(x)\leq \alpha\}$. For a set-valued mapping $F:X\setto Y$, its graph is denoted by $\gph F:=\{(x,y)\in X\times Y:y\in F(x)\}$. For a set $K\subseteq\R^d$, the \emph{indicator function} $\iota_K:\R^d\to\Rex$ is defined by $\iota_K(x):=0$ if $x\in K$ and $\iota_K(x):=+\infty$ otherwise.

 For a proper function $f:\R^d\to\Rex$ and $\gamma>0$, the \emph{Moreau envelope} of $f$ with parameter $\gamma$ is
\[
e_\gamma f(x):=\inf_{z\in\R^d}\Big\{f(z)+\frac{1}{2\gamma}\|z-x\|^2\Big\},
\]
and $f$ is said to be \emph{prox-bounded with threshold $\gamma_f$} if $e_\gamma f>-\infty$ for every $\gamma\in(0,\gamma_f)$.

\subsection{Normal integrands and measurable multifunctions}

For $p\in[1,\infty]$, the notation $L^p_\Pp$ refers to the usual Lebesgue space with respect to~$\Pp$. Let $X$ be a metric space. A set-valued mapping $C:\Xi\setto X$ is called \emph{measurable} if, for every open set $U\subset X$, $\{\xi\in\Xi:C(\xi)\cap U\neq\emp\}\in\mathcal F$. When $X$ is a Polish space and $C$ has nonempty compact values, this is the standard measurability notion used for random closed sets and measurable selection arguments.

A mapping $a:\Xi\times X\to\R^m$ is called a \emph{Carath\'eodory mapping} if $a(\cdot,t)$ is measurable for every $t\in X$ and $a(\xi,\cdot)$ is continuous for every $\xi\in\Xi$. Similarly, a mapping $g:\Xi\times X\times\R^d\to\R$ is Carath\'eodory if $g(\cdot,t,x)$ is measurable for every $(t,x) \in X\times\R^d$ and $g(\xi,\cdot,\cdot)$ is continuous for every $\xi\in\Xi$.

 Given an integrand $\varphi:\Xi\times\R^d\to\Rex$, we use the shorthand
\[
\varphi_\xi(x):=\varphi(\xi,x),
\qquad
\Intf{\varphi}(x):=\int_\Xi \varphi_\xi(x)\,\Pp(d\xi),
\]
whenever the integral is well-defined. An integrand $\varphi:\Xi\times\R^d\to\Rex$ is called a \emph{normal integrand} if the epigraphical mapping $\xi\mapsto \epi \varphi_\xi$ is measurable and has closed values. All normal integrands considered here are proper, in the sense that $\varphi_\xi>-\infty$ on $\R^d$ and $\dom\varphi_\xi\neq\emp$ for every $\xi\in\Xi$. Moreover, $\varphi$ is a \emph{convex (concave) normal integrand} if, in addition, $\varphi_\xi$ is convex (concave) for $\Pp$-almost every $\xi\in\Xi$.

We shall use the following consequence of measurable maximum theorems. If $X$ is Polish, $C:\Xi\setto X$ is measurable with nonempty compact values, and $g:\Xi\times X\times\R^d\to\R$ is Carath\'eodory, then the value function
\[
\varphi_\xi(x):=\min_{t\in C(\xi)} g(\xi,t,x)
\]
is well-defined, the minimum is attained, and the resulting mapping $\varphi:\Xi\times\R^d\to\R$ is a normal integrand (see, e.g., \cite[Chapter 14]{rockafellarwets1998}).  

\subsection{Clarke subdifferential for locally Lipschitz functions}

We adopt generalized differentiation for locally Lipschitz functions; see \cite{clarke1983optimization}. Let $f:\R^d\to\R$ be locally Lipschitz continuous around a given $x\in\R^d$. The \emph{Clarke generalized directional derivative} of $f$ at $x$ in the direction $u\in\R^d$ is
\[
f^\circ(x;u):=
\limsup_{\substack{y\to x\\ t\downarrow 0}}
\frac{f(y+tu)-f(y)}{t},
\]
and the \emph{Clarke subdifferential} of $f$ at $x$ is
\[
\partial f(x):=
\left\{v\in\R^d:
\langle v,u\rangle\leq f^\circ(x;u)
\text{ for all } u\in\R^d
\right\}.
\]
Throughout the paper, whenever $f$ is locally Lipschitz, the notation $\partial f(x)$ refers to the Clarke subdifferential. If $f$ is locally Lipschitz around~$x$, then $\partial f(x)$ is a nonempty, compact, and convex subset of $\R^d$. Moreover, if $f$ is continuously differentiable around~$x$, then $\partial f(x)=\{\nabla f(x)\}$, while if $f$ is convex finite-valued (and thus locally Lipschitz continuous), the Clarke subdifferential coincides with the usual convex subdifferential.

Moreover, a locally Lipschitz continuous function $f:\R^d\to\R$ is said to be
\emph{Clarke regular at $x$} if, for every direction $u\in\R^d$, the one-sided
directional derivative
\[
f'(x;u):=\lim_{t\downarrow 0}\frac{f(x+tu)-f(x)}{t}
\]
exists and satisfies
\[
f'(x;u)=f^\circ(x;u).
\]
The function $f$ is Clarke regular on a set $D$ if it is Clarke regular at every
point of~$D$. Continuously differentiable functions and finite convex functions
are Clarke regular. Notice, however, that Clarke regularity is not invariant
under multiplication by $-1$. For example, the concave function $x\mapsto -|x|$
is not Clarke regular at the origin. The importance of Clarke regularity in our work is twofold: on one hand, it establishes the equivalence among the Fréchet, limiting (Mordukhovich), and Clarke subdifferentials. On the other hand, it yields tight calculus rules that hold with equality.

	The following result is well-known. 
	\begin{lemma}\label{lem:smooth_descent_estimate}
		Let $V\subset\R^d$ be convex, let $h:V\to\R$ be continuously differentiable, and let $\kappa\geq 0$ be such that
		\begin{equation}\label{eq:grad_lip}
			\|\nabla h(z)-\nabla h(w)\|\leq \kappa\|z-w\|
			\qquad\text{for all } z,w\in V.
		\end{equation}
		Then, for all $x,y\in V$,
		\begin{equation}\label{eq:two_sided_descent}
			\big|h(y)-h(x)-\langle \nabla h(x),y-x\rangle\big|
			\leq \frac{\kappa}{2}\|y-x\|^2.
		\end{equation}
		In particular, \eqref{eq:grad_lip} holds whenever $h$ is twice-continuously differentiable on $V$ with $\|\nabla^2 h(z)\|\leq\kappa$ for all $z\in V$.
	\end{lemma}

\subsection{Local Lipschitz moduli and lower- and upper-$C^2$ integrands}

\begin{definition}[Locally Lipschitz integrand with integrable modulus]\label{def:locLip_integrable}
	Let $p\in[1,\infty]$. A normal integrand $\varphi:\Xi\times\R^d\to\Rex$ is \emph{locally Lipschitz around $\bar x$ with $p$-integrable modulus} if there exist a neighborhood $W$ of $\bar x$ and a nonnegative $\kappa\in L^p_\Pp$ such that, for $\Pp$-almost every $\xi\in\Xi$, the function $\varphi_\xi$ is finite on $W$ and
	\begin{equation}\label{eq:LLmodulus}
		|\varphi_\xi(y)-\varphi_\xi(z)|
		\leq \kappa(\xi)\|y-z\|
		\quad
		\text{for all } y,z\in W .
	\end{equation}
	When \eqref{eq:LLmodulus} holds for some measurable $\kappa$, not necessarily integrable, we simply say that $\varphi$ is \emph{locally Lipschitz around $\bar x$ (with measurable modulus)}.  We shall say that these definitions hold on a set $U\subseteq\R^d$ if they hold around every point $\bar x\in U$.
\end{definition}

Next, we recall that the local Lipschitz modulus of a convex function is controlled by a bound on its values. The following result is the corresponding statement for convex normal integrands, with the bound and the modulus required to be $p$-integrable.

\begin{proposition}\label{prop:convex_integrable_values}
	Let $\varphi:\Xi\times\R^d\to\Rex$ be a convex normal integrand. Let $p\in[1,\infty]$, let $\bar x\in\R^d$, and suppose that there exists a convex neighborhood $V$ of $\bar x$ with $\varphi(\cdot,x)\in L^p_\Pp$ for all $x\in V$. Then $\varphi$ is locally Lipschitz around $\bar x$ with $p$-integrable modulus. 
\end{proposition}
\begin{proof}
	Since $\bar x\in\inter V$, we may pick affinely independent points $q_0,\ldots,q_d\in V$ and $r>0$ such that $\B_{2r}[\bar x]\subseteq\conv\{q_i:i=0,\ldots,d\}$, and set
	\[
	\mathfrak{m}(\xi):=\max_{i=0,\dots,d}|\varphi_\xi(q_i)|+|\varphi_\xi(\bar x)| ,
	\]
	which is $p$-integrable because $\varphi(\cdot,q_i)\in L^p_\Pp$ for $i=0,\ldots,d$ and $\varphi(\cdot,\bar x)\in L^p_\Pp$. Let $N\in\mathcal F$ be the $\Pp$-null set of those $\xi$ for which $\mathfrak{m}(\xi)=+\infty$, and fix $\xi\in\Xi\setminus N$. Since $\varphi_\xi$ is convex, its supremum over the simplex $\conv\{q_i:i=0,\ldots,d\}$ is attained at a vertex, and thus
	\begin{align}\label{eq:kappa_convex}
		\varphi_\xi(x)\leq\max_{i=0,\dots,d}\varphi_\xi(q_i)\leq\mathfrak{m}(\xi)+\varphi_\xi(\bar x)
		\qquad\text{for all } x\in\B_{2r}[\bar x] ;
	\end{align}
	that is, $\varphi_\xi$ is bounded above on $\B_{2r}[\bar x]$ by $\varphi_\xi(\bar x)+\mathfrak{m}(\xi)$. Applying \cite[Corollary~2.2.12]{zbMATH01807400} on the ball $\B_{2r}[\bar x]$, with excess $\mathfrak{m}(\xi)$ of this upper bound over the value $\varphi_\xi(\bar x)$ at the center, we obtain
	\[
	|\varphi_\xi(x)-\varphi_\xi(y)|\leq\frac{3}{2r}\,\mathfrak{m}(\xi)\|x-y\|
	\qquad\text{for all }x,y\in\B_{r}(\bar x),
	\]
	which ends the proof.
\end{proof}

For an integrand $\varphi:\Xi\times\R^d\to\R$ that is locally Lipschitz around $x$, we write
\[
\partial \varphi_\xi(x):=\partial(\varphi_\xi)(x).
\]

We next define the classes of normal integrands of interest in this work.

\begin{definition}[Lower- and upper-$C^2$ normal integrands]\label{def:upperC2normal}
	Let $\varphi:\Xi\times\R^d\to\Rex$ be a  normal integrand, let $\bar x\in\R^d$, and let $p\in[1,\infty]$. We say that $\varphi$ is a \emph{lower-$C^2$ normal integrand around $\bar x$ with $p$-integrable (curvature) modulus} if there exist a compact convex neighborhood $V$ of $\bar x$, a nonnegative function $\varrho\in L^p_\Pp$, a Polish space $X$, a measurable set-valued mapping $C:\Xi\setto X$ with nonempty compact values, and a Carath\'eodory mapping $g:\Xi\times X\times\R^d\to\R$ such that, for $\Pp$-almost every $\xi\in\Xi$,
	\begin{equation}\label{eq:smooth_value_representation_lower}
		\varphi_\xi(x)=\max_{t\in C(\xi)} g(\xi,t,x)
		\qquad\text{for all }x\in V,
	\end{equation}
	where, for each $\xi\in\Xi$ and $t\in X$, the function $x\mapsto g(\xi,t,x)$ is twice-continuously differentiable on an open set containing $V$, the mappings $g$, $\nabla_x g$, and $\nabla_x^2 g$ are continuous in $(t,x)$, and 
	\begin{equation}\label{eq:hessian_bound_g_lower}
		\|\nabla_x^2g(\xi,t,x)\|\leq \varrho(\xi)
		\quad
		\text{for all } \xi\in\Xi,\; t\in C(\xi),\; x\in V.
	\end{equation}
	If \eqref{eq:smooth_value_representation_lower}--\eqref{eq:hessian_bound_g_lower} hold for some nonnegative measurable $\varrho$, not necessarily integrable, we simply say that $\varphi$ is a \emph{lower-$C^2$ normal integrand around $\bar x$}.
	
	\rev{Similarly}, $\varphi$ is an \emph{upper-$C^2$ normal integrand around $\bar x$} (with $p$-integrable modulus) if $-\varphi$ is a lower-$C^2$ normal integrand around $\bar x$ (with $p$-integrable modulus). More precisely, this amounts to a representation
	\begin{equation}\label{eq:smooth_value_representation}
		\varphi_\xi(x)=\min_{t\in C(\xi)} g(\xi,t,x)
		\qquad\text{for all }x\in V,
	\end{equation}
	with $C$, $g$ and $\varrho$ as above.
\end{definition}

\section{Characterizations of lower- and upper-$C^2$ normal integrands}\label{sec:main_results}

We next establish a characterization of lower-$C^2$ normal integrands as local stochastic value functions, from which the upper-$C^2$ case follows by a change of sign. The first representation uses quadratic-affine functions, while the second uses smooth functions with an integrable Hessian bound.

\begin{theorem}[Equivalent characterizations of lower-$C^2$ normal integrands]\label{thm:eqC2lower}
	Let $\varphi:\Xi\times\R^d\to\R$ be a finite-valued normal integrand, let $\bar x\in\R^d$, and let $p\in[1,\infty]$. Then the following assertions are equivalent:
	\begin{enumerate}[label=(\alph*)]
		\item\label{it:eqC2l-a} $\varphi$ is a lower-$C^2$ normal integrand around $\bar x$ with $L^p_\Pp$ curvature modulus, i.e., Definition~\ref{def:upperC2normal} holds.
		
		\item\label{it:eqC2l-b} There exist a compact convex neighborhood $V$ of $\bar x$ and a nonnegative function $\varrho\in L^p_\Pp$ such that, for $\Pp$-almost every $\xi\in\Xi$, the function $\varphi_\xi$ is locally Lipschitz on a neighborhood of $V$ and
		\begin{equation}\label{eq:lowerC2_ascent}
			\varphi_\xi(y)
			\geq
			\varphi_\xi(x)+\langle v,y-x\rangle
			-\frac{\varrho(\xi)}{2}\|y-x\|^2
			\quad
			\text{for all } x,y\in V \text{ and all } v\in\partial\varphi_\xi(x).
		\end{equation}

		\item\label{it:eqC2l-c} There exist a compact convex neighborhood $V$ of $\bar x$, a nonnegative function $\varrho\in L^p_\Pp$, a Polish space $X$, a measurable set-valued mapping $C:\Xi\setto X$ with nonempty compact values, and Carath\'eodory mappings $a:\Xi\times X\to\R^d$ and $b:\Xi\times X\to\R$ such that, for $\Pp$-almost every $\xi\in\Xi$ and all $x\in V$,
		\begin{equation}\label{eq:quadratic_affine_representation_lower}
			\varphi_\xi(x)
			=
			\max_{t\in C(\xi)}
			\left\{
			-\frac{\varrho(\xi)}{2}\|x\|^2
			+\langle a(\xi,t),x\rangle
			+b(\xi,t)
			\right\}.
		\end{equation}
		\item\label{it:eqC2l-d} There exist a compact convex neighborhood $V$ of $\bar x$, a nonnegative function $\varrho\in L^p_\Pp$, and a convex normal integrand $f:\Xi\times\R^d\to\R$ such that, for $\Pp$-almost every $\xi\in\Xi$,
		\begin{equation}\label{eq:convex_minus_quadratic}
			\varphi_\xi(x)=f_\xi(x)-\frac{\varrho(\xi)}{2}\|x\|^2
			\qquad\text{for all } x\in V .
		\end{equation}
	\end{enumerate}
	Moreover, the modulus $\varrho$ can be taken to be the same in all items \ref{it:eqC2l-a}--\ref{it:eqC2l-d}.
\end{theorem}

\begin{proof}

We first show that \ref{it:eqC2l-a} implies \ref{it:eqC2l-b}.  Let $V$, $\varrho$, $X$, $C$ and $g$ be as in Definition~\ref{def:upperC2normal}. By \eqref{eq:smooth_value_representation_lower}, the compactness of $C(\xi)\times V$, and the continuity of $\nabla_x g(\xi,\cdot,\cdot)$, each $\varphi_\xi$ is Lipschitz on $V$; hence, shrinking $V$ if necessary, we may assume that each $\varphi_\xi$ is Lipschitz on a neighborhood of $V$. Let $\xi\in\Xi$ and $x,y\in V$, and define the set of active indices $C^*_x(\xi):=\argmax_{t\in C(\xi)}g(\xi,t,x)$, which is nonempty by compactness of $C(\xi)$ and continuity of $g(\xi,\cdot,x)$. In view of \eqref{eq:smooth_value_representation_lower} and the smoothness of $x\mapsto g(\xi,t,x)$ (see, e.g., \cite[Theorem 10.31]{rockafellarwets1998}), for any $v\in\partial\varphi_\xi(x)$ there exist $n\geq 1$, $t_i\in C^*_x(\xi)$ and $\lambda_i\geq 0$, $i=1,\dots,n$, such that $\sum_{i=1}^n\lambda_i=1$ and
	\begin{equation}\label{eq:subgrad-convex-comb-lower}
		v=\sum_{i=1}^n\lambda_i\nabla_x g(\xi,t_i,x).
	\end{equation}
	Furthermore, by \eqref{eq:hessian_bound_g_lower} and Lemma~\ref{lem:smooth_descent_estimate}, for all $i=1,\dots,n$,
	\begin{equation*}
		\varphi_\xi(y)\geq g(\xi,t_i,y)
		\geq
		g(\xi,t_i,x)+\langle\nabla_x g(\xi,t_i,x),y-x\rangle
		-\frac{\varrho(\xi)}{2}\|y-x\|^2.
	\end{equation*}
	Multiplying by $\lambda_i$ and summing over $i=1,\dots,n$ yields
	\begin{align*}
		\varphi_\xi(y)
		&\geq
		\sum_{i=1}^n \lambda_i g(\xi,t_i,x)+\left\langle\sum_{i=1}^n \lambda_i\nabla_x g(\xi,t_i,x),y-x\right\rangle
		-\frac{\varrho(\xi)}{2}\|y-x\|^2\\
		&=\varphi_\xi(x)+\langle v,y-x\rangle-\frac{\varrho(\xi)}{2}\|y-x\|^2,
	\end{align*}
	where in the last line we used that $t_i\in C^*_x(\xi)$, so that $g(\xi,t_i,x)=\varphi_\xi(x)$, together with \eqref{eq:subgrad-convex-comb-lower}. This is \eqref{eq:lowerC2_ascent}.
	
	Next, we prove \ref{it:eqC2l-b} implies \ref{it:eqC2l-c}.
	Let $V$ and $\varrho\in L^p_\Pp$ be as in \ref{it:eqC2l-b}. Set $X:=V\times\R^d$. For each fixed $\xi$, the set \begin{equation}\label{eq:C_subdiff_graph_lower}
		C(\xi):=\gph\partial\varphi_\xi\cap X
	\end{equation}
	is nonempty and compact.  Indeed, $C(\xi)$ is closed due to closedness of $V$ and Lipschitz continuity of $\varphi_\xi$   on a neighborhood of $V$ , which guarantees closedness of the Clarke subdifferential  at the points of $V$  \cite[Proposition~2.1.5(b)]{clarke1983optimization}. Furthermore, $C(\xi)$ is a subset of a compact set. Let $U\supseteq V$ be an open set on which $\varphi_\xi$ is locally Lipschitz. For each $x\in V$, there exists an open ball $B_x\subseteq U$ on which $\varphi_\xi$ is Lipschitz with some constant $K_x$, and by compactness $V\subseteq B_{x_1}\cup\dots\cup B_{x_m}$ for some $m \in \N$. In view of \cite[Proposition~2.1.2]{clarke1983optimization}, we have
\[
\gph \partial\varphi_\xi\cap (V\times\R^d)\subseteq V\times K_\xi\B,
\qquad K_\xi:=\max_{i=1,\dots,m}K_{x_i}<+\infty ,
\] which proves the claim. Furthermore,   
 $C$ is a measurable set-valued mapping by standard measurability properties (see, e.g., \cite[Theorem 14.56]{rockafellarwets1998}) with nonempty compact values. For $t=(z,v)\in X$, define
	\begin{equation}\label{eq:a_b_from_subgrad_lower}
		a(\xi,t):=v+\varrho(\xi)z,
		\qquad
		b(\xi,t):=\varphi_\xi(z)-\frac{\varrho(\xi)}{2}\|z\|^2-\langle v,z\rangle.
	\end{equation}
	The mappings $a$ and $b$ are Carath\'eodory: $a$ is affine in $t$, and since $\varphi_\xi$ is continuous on $V$ and $\varphi$ is a normal integrand, $b$ is continuous in $t=(z,v)$ for each $\xi$ and measurable in $\xi$ for each fixed $t$.

	Fix $\xi\in\Xi$ and $y\in V$. From \eqref{eq:lowerC2_ascent}, for every $t=(z,v)\in C(\xi)$,
	\begin{align*}
		\varphi_\xi(y)
		&\geq
		\varphi_\xi(z)+\langle v,y-z\rangle-
		\frac{\varrho(\xi)}{2}\|y-z\|^2 \\
		&=
		-\frac{\varrho(\xi)}{2}\|y\|^2
		+\langle v+\varrho(\xi)z,y\rangle
		+\varphi_\xi(z)-\frac{\varrho(\xi)}{2}\|z\|^2-
		\langle v,z\rangle \\
		&=
		-\frac{\varrho(\xi)}{2}\|y\|^2+\langle a(\xi,t),y\rangle+b(\xi,t).
	\end{align*}
	Taking the maximum over $t\in C(\xi)$ gives one inequality in \eqref{eq:quadratic_affine_representation_lower}. Conversely, since $\partial\varphi_\xi(y)\neq\emptyset$, pick $v_y\in\partial\varphi_\xi(y)$. Then $t=(y,v_y)\in C(\xi)$ and the corresponding quadratic-affine function equals $\varphi_\xi(y)$ at $y$. Hence \eqref{eq:quadratic_affine_representation_lower} holds.
	
	Now we show that \ref{it:eqC2l-c} implies \ref{it:eqC2l-a}.
	Assume that \eqref{eq:quadratic_affine_representation_lower} holds and define
	\[
	g(\xi,t,x):=-\frac{\varrho(\xi)}{2}\|x\|^2+
	\langle a(\xi,t),x\rangle+b(\xi,t).
	\]
	Then $g$ is Carath\'eodory, $x\mapsto g(\xi,t,x)$ is twice-continuously differentiable, the mappings $g$, $\nabla_x g=-\varrho(\xi)x+a(\xi,t)$ and $\nabla_x^2 g=-\varrho(\xi)I$ are continuous in $(t,x)$, and $\|\nabla^2_x g(\xi,t,x)\|=\varrho(\xi)$. Hence \eqref{eq:smooth_value_representation_lower}--\eqref{eq:hessian_bound_g_lower} hold with the same $V$, $\varrho$, $X$ and $C$, i.e., $\varphi$ is a lower-$C^2$ normal integrand around $\bar x$ with $L^p_\Pp$ curvature modulus.
	
	To show that \ref{it:eqC2l-c} implies \ref{it:eqC2l-d}, observe that
	due to \eqref{eq:quadratic_affine_representation_lower}, it is enough to consider $f(\xi,x):= \max_{t\in C(\xi)}
	\left\{ \langle a(\xi,t),x\rangle
	+b(\xi,t)
	\right\}$, which is clearly a finite convex normal integrand.  
	
	To conclude the proof, we show \ref{it:eqC2l-d} implies \ref{it:eqC2l-b}. If \eqref{eq:convex_minus_quadratic} holds, then $\varphi_\xi+\frac{\varrho(\xi)}{2}\|\cdot\|^2=f_\xi$ is convex on $V$. Shrinking $V$ if necessary, $\varphi_\xi$ is locally Lipschitz on a neighborhood of $V$, since $f_\xi$ is convex and finite on a neighborhood of $V$. For $x,y\in V$ and $v\in\partial\varphi_\xi(x)$ we have $v+\varrho(\xi)x\in\partial f_\xi(x)$ by the smooth sum rule for the Clarke subdifferential, and the convex subgradient inequality gives
	\[
	\varphi_\xi(y)+\frac{\varrho(\xi)}{2}\|y\|^2
	\geq
	\varphi_\xi(x)+\frac{\varrho(\xi)}{2}\|x\|^2
	+\big\langle v+\varrho(\xi)x,\,y-x\big\rangle,
	\]
	which upon rearranging is exactly \eqref{eq:lowerC2_ascent}.
	
\end{proof}

We next record the corresponding characterization of upper-$C^2$ normal integrands, which follows by applying Theorem~\ref{thm:eqC2lower} to $-\varphi$. Observe that the characterization states that upper-$C^2$ normal integrands satisfy a generalization of the so-called \emph{descent lemma} (cf. Lemma~\ref{lem:smooth_descent_estimate}), pivotal in the convergence analysis of optimization algorithms.

\begin{corollary}[Equivalent characterizations of upper-$C^2$ normal integrands]\label{cor:eqC2upper}
	Let $\varphi:\Xi\times\R^d\to\R$ be a finite-valued normal integrand, let $\bar x\in\R^d$, and let $p\in[1,\infty]$. Then the following assertions are equivalent:
	\begin{enumerate}[label=(\alph*)]
		\item\label{it:eqC2u-a} $\varphi$ is an upper-$C^2$ normal integrand around $\bar x$ with $L^p_\Pp$ curvature modulus, i.e., there exist a compact convex neighborhood $V$ of $\bar x$, a nonnegative $\varrho\in L^p_\Pp$, a Polish space $X$, a measurable set-valued mapping $C:\Xi\setto X$ with nonempty compact values, and a Carath\'eodory mapping $g:\Xi\times X\times\R^d\to\R$ satisfying \eqref{eq:hessian_bound_g_lower} and the smoothness requirements of Definition~\ref{def:upperC2normal}, such that
		\begin{equation}\label{eq:smooth_value_representation_upper}
			\varphi_\xi(x)=\min_{t\in C(\xi)} g(\xi,t,x)
			\qquad\text{for all }x\in V.
		\end{equation}
		
		\item\label{it:eqC2u-b} There exist a compact convex neighborhood $V$ of $\bar x$ and a nonnegative function $\varrho\in L^p_\Pp$ such that, for $\Pp$-almost every $\xi\in\Xi$, the function $\varphi_\xi$ is locally Lipschitz on a neighborhood of $V$ and
		\begin{equation}\label{eq:upperC2_descent}
			\varphi_\xi(y)
			\leq
			\varphi_\xi(x)+\langle v,y-x\rangle
			+\frac{\varrho(\xi)}{2}\|y-x\|^2
			\quad
			\text{for all } x,y\in V \text{ and all } v\in\partial\varphi_\xi(x).
		\end{equation}

		\item\label{it:eqC2u-c} There exist a compact convex neighborhood $V$ of $\bar x$, a nonnegative function $\varrho\in L^p_\Pp$, a Polish space $X$, a measurable set-valued mapping $C:\Xi\setto X$ with nonempty compact values, and Carath\'eodory mappings $a:\Xi\times X\to\R^d$ and $b:\Xi\times X\to\R$ such that, for $\Pp$-almost every $\xi\in\Xi$ and all $x\in V$,
		\begin{equation}\label{eq:quadratic_affine_representation_upper}
			\varphi_\xi(x)
			=
			\min_{t\in C(\xi)}
			\left\{
			\frac{\varrho(\xi)}{2}\|x\|^2
			+\langle a(\xi,t),x\rangle
			+b(\xi,t)
			\right\}.
		\end{equation}
		\item\label{it:eqC2u-d} There exist a compact convex neighborhood $V$ of $\bar x$, a nonnegative function $\varrho\in L^p_\Pp$, and a concave normal integrand $f:\Xi\times\R^d\to\R$ such that, for $\Pp$-almost every $\xi\in\Xi$,
		\begin{equation}\label{eq:concave_plus_quadratic}
			\varphi_\xi(x)=f_\xi(x)+\frac{\varrho(\xi)}{2}\|x\|^2
			\qquad\text{for all } x\in V.
		\end{equation}
		
	\end{enumerate}
	Moreover, the modulus $\varrho$ can be taken to be the same in all items \ref{it:eqC2u-a}--\ref{it:eqC2u-d}.
\end{corollary}

In the following, we collect several direct consequences of Corollary~\ref{cor:eqC2upper}, illustrating that the assumptions are easy to check in common stochastic programming constructions.

\begin{example}[Smooth normal integrands]\label{ex:smooth_integrands}
	Let $\varphi:\Xi\times\R^d\to\R$ be a Carath\'eodory integrand. Fix $\bar x\in\R^d$ and suppose that there exist a compact convex neighborhood $V$ of $\bar x$ and nonnegative functions $\kappa\in L^2_\Pp$ and $\varrho\in L^1_\Pp$ such that, for every $\xi\in\Xi$, the function $\varphi_\xi$ is twice-continuously differentiable on (an open set containing) $V$ and
	\[
	\|\nabla_x\varphi_\xi(y)-\nabla_x\varphi_\xi(z)\|
	\leq \varrho(\xi)\|y-z\|,
	\qquad
	|\varphi_\xi(y)-\varphi_\xi(z)|
	\leq \kappa(\xi)\|y-z\|
	\]
	for all $y,z\in V$. Then $\varphi$ is locally Lipschitz around $\bar x$ with square-integrable modulus, and both lower- and upper-$C^2$ around $\bar x$ with integrable curvature modulus $\varrho$, in view of Lemma~\ref{lem:smooth_descent_estimate} and the identity $\partial\varphi_\xi(x)=\{\nabla_x\varphi_\xi(x)\}$.
\end{example}

\begin{example}[Finite lower envelopes of smooth integrands]\label{ex:finite_minima}
	Let $m\in\N$ and let $f_i:\Xi\times\R^d\to\R$, $i=1,\ldots,m$, be Carath\'eodory integrands. Define
	\[
	\varphi_\xi(x):=\min_{1\leq i\leq m} f_i(\xi,x).
	\]
	Assume that, around a point $\bar x$, each $f_i$ satisfies the hypotheses of Example~\ref{ex:smooth_integrands} on the same compact convex neighborhood $V$, with nonnegative curvature moduli $\varrho_i\in L^1_\Pp$ and Lipschitz moduli $\kappa_i\in L^2_\Pp$. Then $\varphi$ is locally Lipschitz around $\bar x$ with square-integrable modulus $\kappa(\xi):=\max_{1\leq i\leq m}\kappa_i(\xi)$ and upper-$C^2$ around $\bar x$ with curvature modulus
	\[
	\varrho(\xi):=\max_{1\leq i\leq m}\varrho_i(\xi).
	\]
	Indeed, the conclusion follows from Corollary~\ref{cor:eqC2upper}\ref{it:eqC2u-a} by taking $X=\{1,\ldots,m\}$ and $C(\xi)=X$.    By replacing the minimum with a maximum in the definition of $\varphi_\xi$, we obtain a lower-$C^2$ normal integrand. 
\end{example}

\begin{example}[A nonsmooth DC sample function]\label{ex:dc_sample}
	Let $A:\Xi\to\R^{m\times d}$ and $c:\Xi\to\R^m$ be measurable, and define
	\[
	\varphi_\xi(x)
	:=\frac12\|A(\xi)x-c(\xi)\|^2-\lambda(\xi)\|x\|_1,
	\]
	where $\lambda:\Xi\to[0,\infty)$ is measurable. On any compact neighborhood $V$ of $\bar x$, this has the form $g_\xi-h_\xi$ with
	\[
	g_\xi(x)=\frac12\|A(\xi)x-c(\xi)\|^2,
	\quad h_\xi(x)=\lambda(\xi)\|x\|_1 \quad \mbox{ for } x \in V.
	\]
	If $\|A(\xi)\|^2\in L^1_\Pp$ and the corresponding local Lipschitz moduli of $g$ and $h$ on $V$ belong to $L^2_\Pp$, then Corollary~\ref{cor:dc_integrands} below shows that $\varphi$ is a locally Lipschitz upper-$C^2$ normal integrand. 
\end{example}

\begin{example}[Two-stage stochastic programs with compact recourse]\label{ex:two_stage}
	Consider a two-stage stochastic program of the form
	\[
	\min_{x\in\R^d}\; \psi(x)+\int_\Xi \varphi_\xi(x)\,\Pp(d\xi),
	\]
	with second-stage value function
	\[
	\varphi_\xi(x):=\min_{y\in C(\xi)} G(\xi,y,x).
	\]
	Assume that $C:\Xi\setto\R^m$ is measurable with nonempty compact values and that $G:\Xi\times\R^m\times\R^d\to\R$ is a Carath\'eodory integrand. Fix $\bar x\in\R^d$ and suppose that there exist a compact convex neighborhood $V$ of $\bar x$ and a nonnegative $\varrho\in L^1_\Pp$ such that, for every $\xi\in\Xi$ and every $y\in C(\xi)$, the mapping $x\mapsto G(\xi,y,x)$ is twice-continuously differentiable on (an open set containing) $V$, the mappings $\nabla_x G(\xi,\cdot,\cdot)$ and $\nabla_x^2 G(\xi,\cdot,\cdot)$ are continuous in $(y,x)$, and
	\[
	\|\nabla_x^2 G(\xi,y,x)\|\leq \varrho(\xi)
	\quad
	\text{for all } x\in V.
	\]
	Then Corollary~\ref{cor:eqC2upper}\ref{it:eqC2u-a} implies that $\varphi$ is an upper-$C^2$ normal integrand around $\bar x$ with integrable curvature modulus. If, in addition, $\varphi$ is locally Lipschitz around $\bar x$ with a square-integrable modulus, the reduced expected-value objective
	\[
	x\mapsto \psi(x)+\Intf{\varphi}(x)
	\]
	falls within the upper-$C^2$ normal-integrand framework developed above, once the deterministic term $\psi$ is handled separately.
\end{example}

\section{Archetypal classes of lower- and upper-$C^2$ normal integrands}\label{sec:examples_classes}

In this section we examine how some operations that are standard in optimization, including difference-of-convex decompositions, Moreau envelopes, and compositions, produce normal integrands that are lower- or upper-$C^2$ in the sense of Definition~\ref{def:upperC2normal}. At the level of a single sample function, each of these facts is classical, and the arguments below follow the deterministic ones closely. What is specific to the present setting is a quantitative   treatment of the random curvature and Lipschitz moduli that these constructions generate. 

In most statements we exhibit the moduli explicitly in terms of the data, and we track how their integrability exponents are produced from those of the data, typically through H\"older's inequality. 
This allows the samplewise estimates to be transferred to the expected functional $\Intf{\varphi}$,  which will be useful in Section~\ref{sec:lipschitz_moduli}. It is also where new hypotheses appear, as a construction may be harmless samplewise, but fail to produce an integrable modulus.

We start by showing that the difference of a Lipschitz smooth function and a convex function is an upper-$C^2$ normal integrand.

\begin{corollary} \label{cor:dc_integrands}
	Let $p,r\in[1,\infty]$ and let $\varphi:\Xi\times\R^d\to\R$ be a normal integrand. Fix $\bar x\in\R^d$ and suppose that there are normal integrands $g:\Xi\times\R^d\to\R$ and $h:\Xi\times\R^d\to\R$, a compact convex neighborhood $V$ of $\bar x$, and nonnegative functions $\varrho\in L^p_\Pp$, $\kappa\in L^r_\Pp$ and $\ell\in L^r_\Pp$ such that:
	
	\[
	\varphi(\xi,x)=g(\xi,x)-h(\xi,x)\quad\text{for all } x\in V\text{ and all } \xi\in\Xi,
	\] and

	\begin{enumerate}[label=(\roman*)]
		\item for every $\xi\in\Xi$, the function $g_\xi:=g(\xi,\cdot)$ is continuously differentiable on $V$ and
		\[
		\|\nabla g_\xi(y)-\nabla g_\xi(z)\|
		\leq \varrho(\xi)\|y-z\|
		\quad\text{for all } y,z\in V;
		\]
		\item for every $\xi\in\Xi$, the function $h_\xi:=h(\xi,\cdot)$ is finite and convex on $V$;
		\item for all $y,z\in V$ and all $\xi\in\Xi$,
		\[
		|g_\xi(y)-g_\xi(z)|\leq \kappa(\xi)\|y-z\|,
		\qquad
		|h_\xi(y)-h_\xi(z)|\leq \ell(\xi)\|y-z\|.
		\]
	\end{enumerate}
	Then $\varphi$ is locally Lipschitz around $\bar x$ with $r$-integrable modulus $\kappa+\ell$ and upper-$C^2$ around $\bar x$ with $p$-integrable curvature modulus $\varrho$.
\end{corollary}

\begin{proof}
	The local Lipschitz estimate follows immediately from~(iii). Let $x,y\in V$ and $v\in\partial\varphi_\xi(x)$. Since $g_\xi$ is continuously differentiable and $h_\xi$ is finite convex, the Clarke calculus rules give
	\[
	\partial\varphi_\xi(x)
	=\nabla g_\xi(x)-\partial h_\xi(x),
	\]
	where $\partial h_\xi(x)$ is the convex subdifferential. Hence there exists $\eta\in\partial h_\xi(x)$ such that $v=\nabla g_\xi(x)-\eta$. By Lemma~\ref{lem:smooth_descent_estimate},
	\[
	g_\xi(y)\leq g_\xi(x)+\langle\nabla g_\xi(x),y-x\rangle
	+\frac{\varrho(\xi)}{2}\|y-x\|^2,
	\]
	while convexity of $h_\xi$ gives
	\[
	h_\xi(y)\geq h_\xi(x)+\langle\eta,y-x\rangle.
	\]
	Subtracting the second inequality from the first yields
	\[
	\varphi_\xi(y)\leq \varphi_\xi(x)+\langle v,y-x\rangle
	+\frac{\varrho(\xi)}{2}\|y-x\|^2,
	\]
	which is \eqref{eq:upperC2_descent}. Therefore, by Corollary~\ref{cor:eqC2upper}, we conclude the result. 
\end{proof}

Next, we show that under appropriate integrability assumptions, the Moreau envelope of a normal integrand is an upper-$C^2$ normal integrand.

\begin{corollary}[Moreau envelope]\label{cor:moreau_upperC2}
	Let $\varphi:\Xi\times\R^d\to\Rex$ be a  normal integrand and let $\bar x\in\R^d$. Let $p\in[2,\infty]$ and $r,s\in[1,\infty]$, and suppose that there exist $\gamma^\varphi>0$, $\alpha:\Xi\to\R^d$ $p$-integrable and   $\beta :\Xi \to \R$  $r$-integrable such that
	\begin{equation}\label{eq:linear-minorant}
		\varphi_\xi(x) \geq -\frac{1}{2\gamma^\varphi}\|x\|^2 + \langle \alpha_\xi, x\rangle + \beta_\xi
		\qquad\text{for all } x\in\R^d,\ \xi\in\Xi,
	\end{equation}
	and that $\varphi(\cdot,\bar x)\in L^s_\Pp$. Then, for every $\gamma\in(0,\gamma^\varphi)$, the function $e_\gamma\varphi$ is a finite-valued upper-$C^2$ normal integrand around $\bar x$ with $p'$-integrable Lipschitz modulus and $q'$-integrable curvature modulus, where
	\[
	p'=\min\Big\{\tfrac p2,\,r,\,s\Big\}
	\qquad\text{and}\qquad
	q'=\infty .
	\]
\end{corollary}

\begin{proof}
	Fix $\gamma\in(0,\gamma^\varphi)$ and set $c_\gamma:=\dfrac1\gamma-\dfrac{1}{\gamma^\varphi}>0$. 
	By \eqref{eq:linear-minorant}, $\varphi_\xi$ is prox-bounded with threshold at least $\gamma^\varphi$, so $e_\gamma\varphi_\xi$ is finite-valued and, by \cite[Proposition 2.7]{aragon2025coderivative}, for all $x\in\R^d$,
	\begin{equation}\label{eq:asplund_decomposition}
		e_\gamma\varphi_\xi(x)=\frac{1}{2\gamma}\|x\|^2-A_\gamma\varphi_\xi(x),
		\qquad
		A_\gamma\varphi_\xi(x):=\sup_{z\in\R^d}\Big\{\frac1\gamma\langle z,x\rangle-\varphi_\xi(z)-\frac{1}{2\gamma}\|z\|^2\Big\}.
	\end{equation}
	Being a pointwise supremum of affine functions of $x$, the function $A_\gamma\varphi_\xi$ is convex and lower semicontinuous. Next, using   \eqref{eq:linear-minorant}, we have that
	\begin{equation}\label{eq:aux-1}
	A_\gamma\varphi_\xi(x)
	\leq
	\sup_{z\in\R^d}\Big\{\Big\langle \tfrac{x}{\gamma}-\alpha_\xi,\,z\Big\rangle-\frac{c_\gamma}{2}\|z\|^2\Big\}-\beta_\xi
	=\frac{1}{2c_\gamma}\Big\|\tfrac{x}{\gamma}-\alpha_\xi\Big\|^2-\beta_\xi .\end{equation} In particular, $A_\gamma\varphi_\xi$ is finite, hence continuous.  For fixed $x$, $(\xi,z)\mapsto\varphi_\xi(z)+\frac{1}{2\gamma}\|z-x\|^2$ is a normal integrand \cite[Proposition~14.44(d)]{rockafellarwets1998}, so $\xi\mapsto e_\gamma\varphi_\xi(x)$ is measurable \cite[Theorem~14.37]{rockafellarwets1998}, and so is $\xi\mapsto A_\gamma\varphi_\xi(x)$ by \eqref{eq:asplund_decomposition}. Hence, $A_\gamma\varphi$ is a Carath\'eodory integrand, and thus a normal integrand \cite[Example~14.29]{rockafellarwets1998}.
    Moreover, evaluating the supremum in \eqref{eq:asplund_decomposition} at $z=\bar x$, \begin{equation}\label{eq:aux-2}
	    A_\gamma\varphi_\xi(x)\geq \frac1\gamma\langle \bar x,x\rangle-\varphi_\xi(\bar x)-\frac{1}{2\gamma}\|\bar x\|^2 .
	\end{equation}
	Combining \eqref{eq:aux-1} and \eqref{eq:aux-2}, $A_\gamma\varphi_\xi$ is finite on $\R^d$ and $A_\gamma\varphi(\cdot,x)\in L^{p'}_\Pp$ for every $x\in\R^d$. Hence, by Proposition~\ref{prop:convex_integrable_values}, $A_\gamma\varphi$ is locally Lipschitz around any point with $p'$-integrable modulus. Finally, Corollary~\ref{cor:dc_integrands} gives the conclusion.
\end{proof}

We next extend the composition rule in \cite[Proposition 11 (ii)]{daniilidis2005filling} to the case of normal integrands. 
We include it for completeness and to account for the integrability of the moduli.

\begin{proposition}[Composition rule]\label{prop:composite}
	Let $r\in[1,\infty]$ and let $p,q\in[r,\infty]$ satisfy $\frac1p+\frac1q=\frac1r$. Let $G:\Xi\times\R^d\to\R^m$ be a Carath\'eodory mapping and let $f:\Xi\times\R^m\to\Rex$ be a convex normal integrand.  Let $\bar x\in\R^d$ and let $V$ be a neighborhood of $\bar x$, and  assume that there exist $M\in L^p_\Pp$ and $L\in L^q_\Pp$ such that, for $\Pp$-almost every $\xi\in\Xi$,
	\begin{enumerate}[label=(\roman*)]
		\item $G_\xi$ is differentiable on an open convex set containing $V$, and $\nabla G_\xi$ is Lipschitz continuous on $V$ with modulus $L_\xi\geq0$. Moreover, $\xi\mapsto\|\nabla G_\xi(\bar x)\|$ belongs to $L^q_\Pp$.
		\item $f_\xi$ is finite and Lipschitz continuous with modulus $M_\xi\geq0$ on an open convex set containing $G_\xi(V)$.
	\end{enumerate}
	Then $\varphi:=f\circ G$, that is $\varphi_\xi:=f_\xi\circ G_\xi$, is a lower-$C^2$ normal integrand around $\bar x$ with $r$-integrable curvature modulus 
	and it is locally Lipschitz around $\bar x$ with $r$-integrable modulus. 
	
\end{proposition}

\begin{proof}
	Shrinking $V$ if necessary, we may assume that $V=\B_\delta[\bar x]$ for some $\delta>0$. By (ii), $f$ is finite and continuous in the second variable on a neighborhood of $G_\xi(V)$, hence Carath\'eodory there. Since $G$ is Carath\'eodory, so is $\varphi$ on $V$, and therefore $\varphi$ is a finite-valued normal integrand around $\bar x$.
	
	Fix $\xi\in\Xi$. By (i), $\|\nabla G_\xi(x)\|\leq\|\nabla G_\xi(\bar x)\|+L_\xi\delta$ for all $x\in V$, so (ii) gives, for $x,y\in V$,
	\[
	|\varphi_\xi(y)-\varphi_\xi(x)|\leq M_\xi\|G_\xi(y)-G_\xi(x)\|\leq \kappa(\xi)\|y-x\| ,
	\qquad
	\kappa(\xi):=M_\xi\big(\|\nabla G_\xi(\bar x)\|+L_\xi\delta\big).
	\]
	Since $\|\nabla G(\cdot,\bar x)\|+\delta L(\cdot)\in L^q_\Pp$ and $M\in L^p_\Pp$, the generalized H\"older inequality yields $\kappa\in L^r_\Pp$ with $\|\kappa\|_r\leq\|M\|_p\big\|\,\|\nabla G(\cdot,\bar x)\|+\delta L(\cdot)\,\big\|_q$. In particular $\varphi$ is locally Lipschitz around $\bar x$ with $r$-integrable modulus.
	
	Let $x,y\in V$ and $v\in\partial\varphi_\xi(x)$. The function $f_\xi$ is convex and finite on a neighborhood of $G_\xi(x)$, hence Clarke regular there, and $G_\xi$ is continuously differentiable  at $x$ by (i). Clarke's chain rule \cite[Theorem 2.3.10]{clarke1983optimization} therefore yields
	\[
	\partial\varphi_\xi(x)=\nabla G_\xi(x)^\top\,\partial f_\xi\big(G_\xi(x)\big),
	\]
	so there exists $u\in\partial f_\xi(G_\xi(x))$ with $v=\nabla G_\xi(x)^\top u$. Moreover $\|u\|\leq M_\xi$, because $f_\xi$ is $M_\xi$-Lipschitz on a neighborhood of $G_\xi(x)$. Convexity of $f_\xi$ then gives
	\[
	\varphi_\xi(y)-\varphi_\xi(x)
	=f_\xi(G_\xi(y))-f_\xi(G_\xi(x))
	\geq\big\langle u,G_\xi(y)-G_\xi(x)\big\rangle .
	\]
	Since $V$ is convex, the fundamental theorem of calculus applied to $t\mapsto G_\xi(x+t(y-x))$ yields
	\[
	G_\xi(y)-G_\xi(x)=\int_0^1\nabla G_\xi\big(x+t(y-x)\big)(y-x)\,dt.
	\]
	Hence, using $v=\nabla G_\xi(x)^\top u$, the Cauchy--Schwarz inequality, (i) and $\|u\|\leq M_\xi$,
	\begin{align*}
		\big\langle u,G_\xi(y)-G_\xi(x)\big\rangle
		&=\langle v,y-x\rangle
		+\int_0^1\Big\langle \big(\nabla G_\xi(x+t(y-x))-\nabla G_\xi(x)\big)^\top u,\,y-x\Big\rangle\,dt\\
		&\geq \langle v,y-x\rangle-\int_0^1 L_\xi t\,\|y-x\|\,M_\xi\,\|y-x\|\,dt
		\\
		&	=\langle v,y-x\rangle-\frac{L_\xi M_\xi}{2}\|y-x\|^2 .
	\end{align*}
	Consequently, combining the last two estimates, \eqref{eq:lowerC2_ascent} holds on $V$ with curvature modulus $\varrho=LM$, which belongs to $L^r_\Pp$ by the generalized H\"older inequality, with $\|\varrho\|_r\leq\|L\|_q\|M\|_p$. Theorem~\ref{thm:eqC2lower} gives the conclusion.
\end{proof}

Given a composite objective $f+g$ with $f$ smooth and $g$ proper, lower semicontinuous and prox-bounded, the \emph{forward--backward envelope} with parameter $\gamma>0$ is the value function associated with the forward-backward minimization subproblem; see \cite{stella2017forward,themelis2018forward}. This makes it a natural merit function for proximal-gradient-type algorithms, and an extension of the Moreau envelope. In the stochastic setting the same construction is performed samplewise: for a Carath\'eodory integrand $f$ and a normal integrand $g$ we define
\begin{equation}\label{eq:fbe_def}
	\varphi^{\mathrm{fb}}_{\gamma,\xi}(x)
	:=\min_{z\in\R^d}
	\Big\{
	f_\xi(x)+\langle\nabla f_\xi(x),z-x\rangle+g_\xi(z)+\frac{1}{2\gamma}\|z-x\|^2
	\Big\}.
\end{equation} 
This function is real-valued and, for $\gamma$ small enough, shares the minimizers and the optimal value of $f+g$. 

The next result shows that, under a uniform prox-boundedness condition on $g$ and a Lipschitz Hessian for $f$, the family \eqref{eq:fbe_def} is an upper-$C^2$ normal integrand with integrable moduli. The second-order regularity of $f$ cannot be weakened to a Lipschitz gradient: for $d=1$, $g\equiv0$ and $f'(x)=1-|x|$ near the origin, \eqref{eq:fbe_def} reduces to $f-\tfrac\gamma2(f')^2$, which has a convex kink at $0$ and is therefore not upper-$C^2$ there.

\begin{corollary}[Forward--backward envelope]\label{cor:fbe_upperC2}
	Let $s\in[1,\infty]$ and let $p,q\in[s,\infty]$ satisfy $\frac1p+\frac1q=\frac1s$. Let $\bar x\in\R^d$, let $\delta>0$, set $V:=\B_\delta[\bar x]$ and $R:=\|\bar x\|+\delta$. Let $f:\Xi\times\R^d\to\R$ be a Carath\'eodory integrand and let $g:\Xi\times\R^d\to\Rex$ be a normal integrand such that:
	\begin{enumerate}[label=(\roman*)]
		\item for every $\xi\in\Xi$, $f_\xi$ is twice-continuously differentiable on an open convex set containing $V$, and
		\[
		\|\nabla^2f_\xi(x)\|\leq L_\xi,
		\qquad
		\|\nabla^2f_\xi(x)-\nabla^2f_\xi(y)\|\leq \Gamma_\xi\|x-y\|
		\qquad\text{for all }x,y\in V,
		\]
		with $L,\Gamma\in L^q_\Pp$;
		\item there exist $\gamma^g>0$ and measurable $\alpha:\Xi\to\R^d$, $\beta:\Xi\to\R$ such that
		\[
		g_\xi(z)\geq-\frac{1}{2\gamma^g}\|z\|^2+\langle\alpha_\xi,z\rangle+\beta_\xi
		\qquad\text{for all } z\in\R^d,\ \xi\in\Xi .
		\]
	\end{enumerate}
	Fix $\gamma\in(0,\gamma^g)$, set $T_\xi:=\mathrm{id}-\gamma\nabla f_\xi$, and assume that the Asplund function $A_\gamma g_\xi$ of \eqref{eq:asplund_decomposition} is Lipschitz continuous with modulus $M_\xi$ on an open convex set containing $T_\xi(V)$, with $M\in L^p_\Pp$. Then $\varphi^{\mathrm{fb}}_\gamma$ given by \eqref{eq:fbe_def} is a finite-valued normal integrand which is upper-$C^2$ around $\bar x$ with $s$-integrable curvature modulus
	\begin{equation}\label{eq:fbe_modulus}
		\varrho(\xi):=\frac1\gamma+L_\xi+R\,\Gamma_\xi+\gamma\,\Gamma_\xi M_\xi,
	\end{equation}
	and locally Lipschitz around $\bar x$ with $s$-integrable modulus
	\begin{equation}\label{eq:fbe_lip}
		\kappa(\xi):=R\Big(\frac1\gamma+L_\xi\Big)+M_\xi\big(1+\gamma L_\xi\big).
	\end{equation}
\end{corollary}

\begin{proof}
	Completing the square in \eqref{eq:fbe_def} gives, for all $x\in\R^d$ (see \cite{patrinos2013proximal}),
	\begin{equation}\label{eq:fbe_moreau_form}
		\varphi^{\mathrm{fb}}_{\gamma,\xi}(x)
		=f_\xi(x)-\frac{\gamma}{2}\|\nabla f_\xi(x)\|^2
		+e_\gamma g_\xi\big(T_\xi(x)\big).
	\end{equation}
	By (ii) and the argument in the proof of Corollary~\ref{cor:moreau_upperC2}, $g_\xi$ is prox-bounded with threshold at least $\gamma^g$, the envelope $e_\gamma g_\xi$ is finite-valued, and \eqref{eq:asplund_decomposition} holds for $g$, with $A_\gamma g$ a finite convex Carath\'eodory integrand. Since $\nabla f$ is Carath\'eodory, \eqref{eq:fbe_moreau_form} shows that $\varphi^{\mathrm{fb}}_\gamma$ is Carath\'eodory, hence a finite-valued normal integrand.
	
	Substituting \eqref{eq:asplund_decomposition} into \eqref{eq:fbe_moreau_form} and expanding $\|T_\xi(x)\|^2$, the terms $\pm\frac\gamma2\|\nabla f_\xi(x)\|^2$ cancel and we obtain
	\begin{equation}\label{eq:fbe_dc}
		\varphi^{\mathrm{fb}}_{\gamma,\xi}=\Phi_\xi-\Psi_\xi
		\quad\text{on } V,
		\qquad
		\begin{cases}
			\Phi_\xi(x):=f_\xi(x)+\dfrac{1}{2\gamma}\|x\|^2-\langle x,\nabla f_\xi(x)\rangle,\\[2mm]
			\Psi_\xi:=A_\gamma g_\xi\circ T_\xi .
		\end{cases}
	\end{equation} On the one hand, by (i), $\Phi_\xi$ is continuously differentiable on $V$ with
	\[
	\nabla \Phi_\xi(x)=\frac{1}{\gamma}x-\nabla^2f_\xi(x)\,x ,
	\]
	so that, for $x,y\in V$,
	\[
	\|\nabla \Phi_\xi(x)-\nabla \Phi_\xi(y)\|
	\leq\frac1\gamma\|x-y\|+\|\nabla^2f_\xi(x)-\nabla^2f_\xi(y)\|\,\|x\|+\|\nabla^2f_\xi(y)\|\,\|x-y\|
	\leq\varrho^1(\xi)\|x-y\|,
	\]
	with $\varrho^1(\xi):=\frac1\gamma+L_\xi+R\Gamma_\xi$. Moreover $\|\nabla \Phi_\xi(x)\|\leq R(\frac1\gamma+L_\xi)$ on $V$. 
	
	On the other hand, by (i), $T_\xi$ is Lipschitz on $V$ with modulus $1 + \gamma L_\xi$, and continuously differentiable on a neighborhood of $V$ with $\nabla T_\xi=I-\gamma\nabla^2f_\xi$. Hence $\|\nabla T_\xi\|\leq1+\gamma L_\xi$ on $V$ and $\nabla T_\xi$ is Lipschitz on $V$ with modulus $\gamma\Gamma_\xi$. Since $1+\gamma L\in L^q_\Pp$, $\gamma\Gamma\in L^q_\Pp$ and $M\in L^p_\Pp$, Proposition~\ref{prop:composite} applies with inner mapping $G:=T$ and outer convex integrand $A_\gamma g$, and yields that $\Psi_\xi$ is Lipschitz on $V$ with modulus $M_\xi(1+\gamma L_\xi)$ and lower-$C^2$ on $V$ with $s$-integrable modulus $\sigma(\xi):=\gamma\Gamma_\xi M_\xi$; that is, $\Psi_\xi+\frac{\sigma(\xi)}{2}\|\cdot\|^2$ is convex on $V$ for every $\xi\in\Xi$.
	
	Next, rewrite \eqref{eq:fbe_dc} as
	\[
	\varphi^{\mathrm{fb}}_{\gamma,\xi}
	=\Big[\Phi_\xi+\frac{\sigma(\xi)}{2}\|\cdot\|^2\Big]
	-\Big[\Psi_\xi+\frac{\sigma(\xi)}{2}\|\cdot\|^2\Big].
	\]
	The first bracket is continuously differentiable on $V$ with gradient Lipschitz of modulus $\varrho^1(\xi)+\sigma(\xi)=\varrho(\xi)$, the second is finite and convex on $V$, and both are Lipschitz on $V$. Corollary~\ref{cor:dc_integrands} then gives the upper-$C^2$ property with curvature modulus \eqref{eq:fbe_modulus}. The local Lipschitz modulus \eqref{eq:fbe_lip} follows directly from \eqref{eq:fbe_dc}, since $\Phi_\xi$ is Lipschitz on $V$ with modulus $R(\frac1\gamma+L_\xi)$ and $\Psi_\xi$ with modulus $M_\xi(1+\gamma L_\xi)$.Finally, $\Pp$ is a probability measure and $q\geq s$, so $L,\Gamma\in L^s_\Pp$, while $\Gamma M$ and $LM$ belong to $L^s_\Pp$ by the generalized H\"older inequality. Thus, $\varrho,\kappa\in L^s_\Pp$.
\end{proof}

\begin{remark}[On the Lipschitz assumption on $A_\gamma g$ in Corollary~\ref{cor:fbe_upperC2}]\label{rem:asplund_lip}
Let \textup{(ii)} of Corollary~\ref{cor:fbe_upperC2} hold, and let $\gamma\in(0,\gamma^g)$. Then, for every set $S\subseteq\R^d$, the function $A_\gamma g_\xi$ is Lipschitz on $S$ with modulus
\begin{equation}\label{eq:Asp-Lip}
    \frac1\gamma\,\sup\big\{\|z\|:\ z\in P_\gamma g_\xi(w),\ w\in S\big\},
\end{equation} where $P_\gamma g_\xi(w):=\argmin_{z}\{g_\xi(z)+\frac{1}{2\gamma}\|z-w\|^2\}$, $w \in \R^d$. First, observe that $P_\gamma g_\xi(w) \neq \emptyset$ for every $w\in\R^d$. Indeed, setting $c_\gamma:=\frac1\gamma-\frac1{\gamma^g}>0$, for all $z\in\R^d$ we have
	\begin{equation}\label{eq:moreau_lower}
		g_\xi(z)+\frac{1}{2\gamma}\|z-w\|^2\ \geq\ \frac{c_\gamma}{2}\|z\|^2+\Big\langle\alpha_\xi-\frac w\gamma,z\Big\rangle+\beta_\xi+\frac{1}{2\gamma}\|w\|^2 .
	\end{equation}
	Since $c_\gamma>0$, the lower semicontinuous function on the left-hand side is level-bounded in $z$, and thus its set of minimizers $P_\gamma g_\xi(w)$ is nonempty. Next,  the supremum defining $A_\gamma g_\xi(w')$ is attained at any $z'\in P_\gamma g_\xi(w')$. Hence \[A_\gamma g_\xi(w')-A_\gamma g_\xi(w)\leq\frac1\gamma\langle z',w'-w\rangle\leq\frac1\gamma\|z'\|\,\|w'-w\|,\] and the claim follows by symmetry.


In particular, we have the following two important examples.
\begin{enumerate}[label=(\arabic*)]
\item  Let $g_\xi=\iota_{K(\xi)}$, the indicator function of $K(\xi)$, where $K:\Xi\setto\R^d$ is measurable with nonempty closed values and $K(\xi)\subseteq k_\xi\B$ with $k\in L^p_\Pp$. Then Corollary~\ref{cor:fbe_upperC2}  (ii) holds with $\alpha\equiv0$, $\beta\equiv0$ and any $\gamma^g>0$. Since $P_\gamma g_\xi(w)$ is the set of projections of $w$ onto $K(\xi)$, which lie in $K(\xi)$, it follows from \eqref{eq:Asp-Lip} that $A_\gamma g_\xi$ is Lipschitz on all of $\R^d$ with modulus $M_\xi=k_\xi/\gamma$.
\item  Let $g_\xi$ be Lipschitz on $\R^d$ with modulus $\ell_\xi$, where $\ell\in L^p_\Pp$ (e.g., $g_\xi=\lambda_\xi\|\cdot\|_1$ and minimax concave penalties \cite{MR2604701,MR1946581}). Since $|g_\xi(z)-g_\xi(0)|\leq\ell_\xi\|z\|\leq\frac{1}{2\gamma^g}\|z\|^2+\frac{\gamma^g}{2}\ell_\xi^2$, Corollary~\ref{cor:fbe_upperC2} (ii) holds with $\alpha\equiv0$, $\beta_\xi=g_\xi(0)-\frac{\gamma^g}{2}\ell_\xi^2$ and any $\gamma^g>0$. For $w\in \R^d$ and $z\in P_\gamma g_\xi(w)$, \[\frac{1}{2\gamma}\|z-w\|^2\le g_\xi(w)-g_\xi(z)\le\ell_\xi\|z-w\|,\]so \begin{equation} \label{eq:cota}
    \|z\|\le\|w\|+2\gamma\ell_\xi \quad \mbox{for all } w \in \R^d, \, z\in P_\gamma g_\xi(w).
\end{equation} Let $\delta >0$ and set $K_\xi := (1+\gamma L_\xi)\delta + \|\bar x\| + \gamma \|\nabla f_\xi(\bar x)\|$. Since $T_\xi$ is Lipschitz on $\B_\delta[\bar x]$ with modulus $1+\gamma L_\xi$ by assumption (i) of Corollary~\ref{cor:fbe_upperC2}, for $y = T_\xi(w')$ with $w' \in \B_\delta[\bar x]$ we have \begin{align*}
    \|y\| &\leq \| T_\xi(w') - T_\xi(\bar x) \| + \| T_\xi(\bar x) \|  \leq (1+\gamma L_\xi) \|w'-\bar x\| + \|\bar x\| + \gamma \|\nabla f_\xi(\bar x)\| \leq K_\xi ,
\end{align*} and thus $T_\xi(\B_\delta[\bar x]) \subseteq \B_{K_\xi}[0] \subseteq \B_{K_\xi+1}(0)$. Furthermore, \eqref{eq:Asp-Lip} and \eqref{eq:cota} imply that $A_\gamma g_\xi$ is Lipschitz on the open convex set $\B_{K_\xi+1}(0)$ with modulus $M_\xi=(K_\xi+1)/\gamma+2\ell_\xi$, which lies in $L^p_\Pp$ provided that $\|\nabla f(\cdot,\bar x)\|, L \in L^p_\Pp$.
\end{enumerate}
\end{remark}

The last result of this section extends Proposition~\ref{prop:composite} to
compositions whose inner mapping is no longer smooth, but has lower-$C^2$
components. Smoothness of the inner mapping is replaced by a monotonicity
requirement on the outer convex function, which guarantees that the ascent
inequalities of the components are preserved under composition. This is the (stochastic)
lower-$C^2$ counterpart of the classical rule that a nondecreasing convex function
of convex functions is convex. It is also the structure of the
conditional value-at-risk studied in Section~\ref{sec:app}, where a nondecreasing
convex function is composed with a nonsmooth loss.

\begin{proposition}[Monotone convex composition]\label{prop:composite_monotone}
	Let $r\in[1,\infty]$ and let $p,q\in[r,\infty]$ satisfy $\frac1p+\frac1q=\frac1r$. Let
	$g=(g_1,\dots,g_m):\Xi\times\R^d\to\R^m$, where each $g_j$ is a finite-valued normal
	integrand, and let $f:\Xi\times\R^m\to\R$ be a finite-valued convex normal integrand such
	that, for every $\xi\in\Xi$, $f_\xi$ is nondecreasing with respect to the componentwise
	order of $\R^m$. Let $\bar x\in\R^d$ and assume that there exist a compact convex
	neighborhood $V$ of $\bar x$, nonnegative functions $\varrho_1,\dots,\varrho_m\in L^q_\Pp$
	and a nonnegative $M\in L^p_\Pp$ such that, for $\Pp$-almost every $\xi\in\Xi$,
	\begin{enumerate}[label=(\roman*)]
		\item for each $j=1,\dots,m$, $g_j$ is locally Lipschitz on a neighborhood of $V$
		and lower-$C^2$ around $\bar x$, with the representation
		\eqref{eq:smooth_value_representation_lower} holding on $V$ with curvature
		modulus $\varrho_j$;
		\item $f_\xi$ is Lipschitz continuous with modulus $M_\xi$ on an open convex set
		containing $g_\xi(V)$.
	\end{enumerate}
	Then $\varphi:=f\circ g$ is a lower-$C^2$ normal integrand around $\bar x$ with
	$r$-integrable curvature modulus $\sigma(\xi):=M_\xi\,\|\varrho(\xi)\|$, where
	$\varrho:=(\varrho_1,\dots,\varrho_m)$. If, in addition, each $g_j$ is Lipschitz on $V$
	with modulus $\kappa_j\in L^q_\Pp$, then $\varphi$ is locally Lipschitz around $\bar x$
	with $r$-integrable modulus $M\|\kappa\|$, $\kappa:=(\kappa_1,\dots,\kappa_m)$.
\end{proposition}

\begin{proof}
Due to \cite[Proposition~14.45(b)]{rockafellarwets1998},  $\varphi$ is a normal integrand. 
	Moreover, it is
	locally Lipschitz continuous on a neighborhood of $V$ as the composition of locally Lipschitz maps.

    Next, fix $\xi \in \Xi$, $x,y\in V$ and $v\in\partial\varphi_\xi(x)$.
	By Clarke's chain rule \cite[Theorem~2.3.9]{clarke1983optimization},
	\[
	v\in\conv\Big\{\textstyle\sum_{j=1}^m\alpha_j\zeta_j:\ \alpha\in\partial f_\xi(g_\xi(x)),\
	\zeta_j\in\partial (g_j)_\xi(x)\Big\},
	\]
	so $v=\sum_{k=1}^N\lambda_k w_k$ for some $N \in \N$, $\lambda_k\geq0$, $\sum_{k=1}^N\lambda_k=1$, and
	$w_k=\sum_{j=1}^m\alpha^k_j\zeta^k_j$ for some $\alpha^k\in\partial f_\xi(g_\xi(x))$ and
	$\zeta^k_j\in\partial(g_j)_\xi(x)$. By assumption (ii), $\|\alpha^k\|\leq M_\xi$. Using monotonicity and the convex
	subgradient inequality at $z:=g_\xi(x)$, for $t>0$ and $e_j$ the $j$-th canonical vector in $\R^m$, \[f_\xi(z) \geq f_\xi(z - te_j) \geq f_\xi(z) - t\alpha_j^k ,\] so that  $\alpha_j^k\geq0$ for every $j =1, \dots, m$. Furthermore, by Theorem~\ref{thm:eqC2lower}, assumption (i) implies that \eqref{eq:lowerC2_ascent} holds on $V$ with modulus $\varrho_j(\xi)$, and thus
	\begin{align*}
		\varphi_\xi(y)-\varphi_\xi(x)
		&\geq\big\langle\alpha^k,g_\xi(y)-g_\xi(x)\big\rangle
		\geq\sum_{j=1}^m\alpha^k_j\Big(\langle\zeta^k_j,y-x\rangle
		-\frac{\varrho_j(\xi)}2\|y-x\|^2\Big)\\
    &= \langle w_k,y-x\rangle-\sum_{j=1}^m\frac{\alpha^k_j\varrho_j(\xi)}2\|y-x\|^2 \geq\langle w_k,y-x\rangle-\frac{M_\xi\|\varrho(\xi)\|}{2}\|y-x\|^2 .
	\end{align*}
	Multiplying by $\lambda_k$ and summing over $k= 1, \dots, N $ gives \eqref{eq:lowerC2_ascent} on $V$
	with modulus $\sigma=M\|\varrho\|$. Furthermore,  by the generalized H\"older
	inequality, it follows that $\|\sigma\|_r\leq\|M\|_p\,\|\,\|\varrho\|\,\|_q$, so that $\sigma\in L^r_\Pp$. The conclusion follows
	from Theorem~\ref{thm:eqC2lower}. Finally, for $x,y\in V$,
	$|\varphi_\xi(y)-\varphi_\xi(x)|\leq M_\xi\|g_\xi(y)-g_\xi(x)\|\leq
	M_\xi\|\kappa(\xi)\|\,\|y-x\|$, and again $M\|\kappa\|\in L^r_\Pp$ by the H\"older inequality.
\end{proof}

\section{Integrability of Lipschitz moduli, values and  subgradients}\label{sec:lipschitz_moduli}

In this section we show that, for lower- and upper-$C^2$ normal integrands, the integrability of the values, of the Lipschitz moduli, and of the Clarke subgradients are all related. In particular, 
we prove some equivalences 
that make the samplewise estimates of Section~\ref{sec:examples_classes} usable for the expected functional \eqref{eq:expected_objective_intro}.

The first result 
extends Proposition~\ref{prop:convex_integrable_values} beyond the convex case. For lower- and upper-$C^2$ normal integrands, $p$-integrability of the values on a neighborhood of the reference point is both 
necessary 
and sufficient for local Lipschitz continuity with a $p$-integrable modulus.

\begin{corollary}  \label{cor:lipschitz_from_integrable_values}
Let $p\in[1,\infty]$, let $\varphi:\Xi\times\R^d\to\R$ be a finite-valued normal integrand, and let $\bar x\in\R^d$. Assume that $\varphi$ is either lower-$C^2$ or upper-$C^2$ around $\bar x$ with $p$-integrable curvature modulus $\varrho$. Then the following assertions are equivalent:
\begin{enumerate}[label=(\alph*)]
	\item\label{it:liv-a} there exists a neighborhood $V$ of $\bar x$ such that $\varphi(\cdot,x)\in L^p_\Pp$ for every $x\in V$;
	\item\label{it:liv-b} $\varphi(\cdot,\bar x)\in L^p_\Pp$ and $\varphi$ is locally Lipschitz around $\bar x$ with $p$-integrable Lipschitz modulus.
\end{enumerate}
\end{corollary}
\begin{proof}
Both \ref{it:liv-a} and \ref{it:liv-b} are invariant under the change of sign $\varphi\mapsto-\varphi$, which interchanges the lower- and upper-$C^2$ properties while preserving the curvature modulus. We may therefore assume that $\varphi$ is lower-$C^2$ around $\bar x$.

Suppose that \ref{it:liv-a} holds. Shrinking $V$ if necessary, we may assume that it is a compact convex neighborhood of $\bar x$ on which, by Theorem~\ref{thm:eqC2lower}\ref{it:eqC2l-d}, the representation
\[
	\varphi_\xi(x)=f_\xi(x)-\frac{\varrho(\xi)}{2}\|x\|^2
	\qquad\text{for all } x\in V
\]
holds with a convex normal integrand $f$. Since $\varrho\in L^p_\Pp$ and $V$ is bounded, the function $\xi\mapsto\tfrac{\varrho(\xi)}{2}\|x\|^2$ belongs to $L^p_\Pp$ for every $x\in V$. Hence, $f(\cdot,x)\in L^p_\Pp$ for every $x\in V$. Proposition~\ref{prop:convex_integrable_values} then shows that $f$ is locally Lipschitz around $\bar x$ with $p$-integrable Lipschitz modulus. As $x\mapsto\tfrac{\varrho(\xi)}{2}\|x\|^2$ is Lipschitz on $V$ with modulus $\varrho(\xi)\sup_{u\in V}\|u\|$, which is again $p$-integrable, the same is true of $\varphi$. Finally, $\varphi(\cdot,\bar x)\in L^p_\Pp$ because $\bar x\in V$, so \ref{it:liv-b} holds.

Conversely, suppose that \ref{it:liv-b} holds, and let $V$ be a neighborhood of $\bar x$ on which $\varphi_\xi$ is Lipschitz with a $p$-integrable modulus $\kappa$. Then
\[
	|\varphi_\xi(x)|\leq|\varphi_\xi(\bar x)|+\kappa(\xi)\|x-\bar x\|
	\qquad\text{for all } x\in V ,
\]
and the right-hand side is $p$-integrable for each fixed $x\in V$, since $\varphi(\cdot,\bar x)\in L^p_\Pp$ and $\kappa\in L^p_\Pp$. Hence $\varphi(\cdot,x)\in L^p_\Pp$ for every $x\in V$, which is \ref{it:liv-a}.
\end{proof}
    
Combining Corollary~\ref{cor:lipschitz_from_integrable_values} with the ascent inequality \eqref{eq:lowerC2_ascent} yields a criterion in which the $p$-integrability of the values on a whole neighborhood is replaced by a one-sided bound together with a condition at the reference point alone, namely, the minimum-norm Clarke subgradient at $\bar x$ must be $p$-integrable.

\begin{corollary}\label{cor:lip_modulus_subgrad}
Let $p\in[1,\infty]$, let $\varphi:\Xi\times\R^d\to\R$ be a finite-valued normal integrand, and let $\bar x\in\R^d$ be such that $\varphi(\cdot,\bar x)\in L^p_\Pp$. Assume that, for some $\varepsilon>0$ and with $V:=\B_\varepsilon[\bar x]$, one of the following conditions holds:
\begin{enumerate}[label=(\roman*)]
	\item\label{it:lms-lower} $\varphi$ is a lower-$C^2$ normal integrand around $\bar x$ with $p$-integrable curvature modulus $\varrho$, the ascent inequality \eqref{eq:lowerC2_ascent} holds on $V$, and $\xi\mapsto\sup_{u\in V}\varphi_\xi(u)$ belongs to $L^p_\Pp$;
	\item\label{it:lms-upper} $\varphi$ is an upper-$C^2$ normal integrand around $\bar x$ with $p$-integrable curvature modulus $\varrho$, the descent inequality \eqref{eq:upperC2_descent} holds on $V$, and $\xi\mapsto\inf_{u\in V}\varphi_\xi(u)$ belongs to $L^p_\Pp$.
\end{enumerate}
Then $\varphi$ is locally Lipschitz around $\bar x$ with $p$-integrable Lipschitz modulus if and only if
\begin{equation}\label{eq:subgrad_integrable_at_bar_x}
	\xi\mapsto \dist\big(0,\partial\varphi_\xi(\bar x)\big)
	\quad\text{belongs to } L^p_\Pp .
\end{equation}
\end{corollary}

\begin{proof}
Necessity is immediate: if $\varphi$ is locally Lipschitz around $\bar x$ with $p$-integrable Lipschitz modulus $\kappa$, then $\|v\|\leq\kappa(\xi)$ for every $\xi\in\Xi$ and every $v\in\partial\varphi_\xi(\bar x)$, so that $\dist(0,\partial\varphi_\xi(\bar x))\leq\kappa(\xi)$ and \eqref{eq:subgrad_integrable_at_bar_x} follows.

For the converse, observe that \ref{it:lms-lower} and \ref{it:lms-upper} may be exchanged by the change of sign $\varphi\mapsto-\varphi$, which affects neither \eqref{eq:subgrad_integrable_at_bar_x} nor the conclusion, since $\dist(0,\partial(-\varphi_\xi)(\bar x))=\dist(0,\partial\varphi_\xi(\bar x))$. We may therefore assume without loss of generality that \ref{it:lms-lower} holds.

Suppose that \eqref{eq:subgrad_integrable_at_bar_x} holds, fix $\xi\in\Xi$, and let $\Pi(\bar x,\xi)$ be the projection of $0$ onto the nonempty compact convex set $\partial\varphi_\xi(\bar x)$, so that $\|\Pi(\bar x,\xi)\|=\dist(0,\partial\varphi_\xi(\bar x))$. Applying \eqref{eq:lowerC2_ascent} at $x=\bar x$ with $y\in V$ and $v=\Pi(\bar x,\xi)$, we obtain
\[
	\varphi_\xi(\bar x)
	-\dist\big(0,\partial\varphi_\xi(\bar x)\big)\varepsilon
	-\frac{\varrho(\xi)}{2}\varepsilon^2
	\;\leq\;
	\varphi_\xi(y)
	\;\leq\;
	\sup_{u\in V}\varphi_\xi(u)
	\qquad\text{for all } y\in V .
\]
Both outer bounds are independent of $y$ and $p$-integrable, so there exists $r\in L^p_\Pp$ with $|\varphi_\xi(y)|\leq r(\xi)$ for all $y\in V$. In particular $\varphi(\cdot,y)\in L^p_\Pp$ for every $y\in V$, and Corollary~\ref{cor:lipschitz_from_integrable_values} gives the conclusion.
\end{proof}

The previous result suggests studying the relation between integrability of the subgradients and Lipschitz continuity of the integrand \rev{over open convex sets}. For lower- and upper-$C^2$ normal integrands, \rev{the next result states that} a pointwise $p$-integrability condition on the minimum-norm Clarke subgradient already forces a locally uniform $p$-integrable bound on the entire subdifferential. Note also that the condition is imposed at each point separately, whereas the conclusion is local Lipschitz continuity with a single modulus valid on a whole neighborhood, and the $p$-integrability of the values is needed at one point only.

\begin{corollary}\label{cor:lowerC2_integrand_lip}
Let $p\in[1,\infty]$, let $\varphi:\Xi\times\R^d\to\R$ be a finite-valued normal integrand, and let $U\subseteq\R^d$ be an open convex set containing a point $\bar x$ with $\varphi(\cdot,\bar x)\in L^p_\Pp$. Assume that one of the following conditions holds:
\begin{enumerate}[label=(\roman*)]
	\item\label{it:lcl-lower} $\varphi$ is lower-$C^2$ around every point of $U$ with $p$-integrable curvature modulus;
	\item\label{it:lcl-upper} $\varphi$ is upper-$C^2$ around every point of $U$ with $p$-integrable curvature modulus.
\end{enumerate}
Then the following assertions are equivalent:
\begin{enumerate}[label=(\alph*)]
	\item\label{it:lcl-a} for every $x\in U$,
	\begin{equation}\label{eq_X}
		\E\left[
		\dist\big(0,\partial\varphi_\xi(x)\big)^p
		\right]
		<+\infty ;
	\end{equation}
	\item\label{it:lcl-b} $\varphi$ is locally Lipschitz continuous around every point of $U$ with $p$-integrable Lipschitz modulus.
\end{enumerate}
\end{corollary}

  \begin{proof}
As before, both \ref{it:lcl-a} and \ref{it:lcl-b} are invariant under the change of sign $\varphi\mapsto-\varphi$, which interchanges \ref{it:lcl-lower} and \ref{it:lcl-upper}. We may therefore assume without loss of generality that $\varphi$ is lower-$C^2$ around every point of $U$.

Suppose that \ref{it:lcl-a} holds and consider the set
\begin{equation}\label{def:U}
	\tilde U:=\big\{x\in U:\ \varphi(\cdot,x)\in L^p_\Pp\big\},
\end{equation}
which is nonempty because $\bar x\in\tilde U$. For $w\in U$, let $\Pi(w,\xi)$ denote, as in the proof of Corollary~\ref{cor:lip_modulus_subgrad}, the projection of $0$ onto the nonempty compact convex set $\partial\varphi_\xi(w)$, so that $\|\Pi(w,\xi)\|=\dist(0,\partial\varphi_\xi(w))$. By \eqref{eq_X}, the map $\xi\mapsto\|\Pi(w,\xi)\|$ belongs to $L^p_\Pp$ for every fixed $w\in U$.

\rev{First, we show that} $\varphi$ is locally Lipschitz around every point of \rev{the open set} $\tilde U$ with $p$-integrable Lipschitz modulus. Fix $\hat x\in\tilde U$ and let $W\subseteq U$ be a compact convex neighborhood of $\hat x$ on which \eqref{eq:lowerC2_ascent} and \eqref{eq:convex_minus_quadratic} hold with a nonnegative modulus $\varrho\in L^p_\Pp$ and a convex normal integrand $f$. Let $e_0,\ldots,e_d \in \R^d$ be affinely independent points such that $\Delta:=\conv\{e_0,\ldots,e_d\}\subseteq W$ and $\hat x\in\inter\Delta$.

Applying \eqref{eq:lowerC2_ascent} first at $x=\hat x$, $y=e_i$, $v=\Pi(\hat x,\xi)$, and then at $x=e_i$, $y=\hat x$, $v=\Pi(e_i,\xi)$, we obtain, for each $i=0,1,\ldots,d$,
\[
	\varphi_\xi(\hat x)
	+\big\langle \Pi(\hat x,\xi),e_i-\hat x\big\rangle
	-\frac{\varrho(\xi)}{2}\|e_i-\hat x\|^2
	\leq
	\varphi_\xi(e_i)
	\leq
	\varphi_\xi(\hat x)
	+\big\langle \Pi(e_i,\xi),e_i-\hat x\big\rangle
	+\frac{\varrho(\xi)}{2}\|e_i-\hat x\|^2 .
\]
Since $\varphi(\cdot,\hat x)\in L^p_\Pp$, it follows that $\varphi(\cdot,e_i)\in L^p_\Pp$ for every $i=0,1,\ldots,d$.

Let now $w\in\Delta$, written as a convex combination $w=\sum_{i=0}^d\lambda_ie_i$ with $\lambda_i\geq0$ and $\sum_i\lambda_i=1$. By \eqref{eq:convex_minus_quadratic} and the convexity of $f_\xi$,
\[
	\varphi_\xi(w)
	=f_\xi(w)-\frac{\varrho(\xi)}{2}\|w\|^2
	\leq\sum_{i=0}^d\lambda_if_\xi(e_i)
	=\sum_{i=0}^d\lambda_i\Big(\varphi_\xi(e_i)+\frac{\varrho(\xi)}{2}\|e_i\|^2\Big),
\]
while \eqref{eq:lowerC2_ascent} at $x=\hat x$, $y=w$ and $v=\Pi(\hat x,\xi)$ yields
\[
	\varphi_\xi(w)
	\geq
	\varphi_\xi(\hat x)
	+\big\langle \Pi(\hat x,\xi),w-\hat x\big\rangle
	-\frac{\varrho(\xi)}{2}\|w-\hat x\|^2 .
\]
Both bounds of $\varphi_\xi(w)$ are dominated by $p$-integrable functions independent of $w$, so $\varphi(\cdot,w)\in L^p_\Pp$ for every $w\in\Delta$. Hence $\Delta\subseteq\tilde U$, and since $\hat x\in\inter\Delta$, the set $\tilde U$ is a neighborhood of $\hat x$. Moreover, $\xi\mapsto\sup_{u\in\Delta}\varphi_\xi(u)$ is $p$-integrable, so \rev{in view of \ref{it:lcl-a},} Corollary~\ref{cor:lip_modulus_subgrad} shows that $\varphi$ is locally Lipschitz around $\hat x$ with $p$-integrable Lipschitz modulus.

\rev{We next show that} $\tilde U=U$. Suppose, on the contrary, that $U\setminus\tilde U\neq\emp$. Since $U$ is nonempty, open, and convex, hence connected,  
the set $\tilde U$ cannot be relatively closed in $U$. Therefore, there exists $y\in U\cap\bd\tilde U$, and $y\notin\tilde U$ because $\tilde U$ is open. As $\varphi$ is lower-$C^2$ around $y$, there exist a neighborhood $V\subseteq U$ of $y$ and a nonnegative $\varrho\in L^p_\Pp$ such that \eqref{eq:lowerC2_ascent} holds on $V$. Since $y\in\bd\tilde U$, we may pick $\hat x\in\tilde U\cap V$. Applying \eqref{eq:lowerC2_ascent} at $x=\hat x$, $y$ and then at $x=y$, $\hat x$ gives
\[
	\varphi_\xi(\hat x)
	+\big\langle \Pi(\hat x,\xi),y-\hat x\big\rangle
	-\frac{\varrho(\xi)}{2}\|y-\hat x\|^2
	\leq
	\varphi_\xi(y)
	\leq
	\varphi_\xi(\hat x)
	+\big\langle \Pi(y,\xi),y-\hat x\big\rangle
	+\frac{\varrho(\xi)}{2}\|y-\hat x\|^2 ,
\]
thus $\varphi(\cdot,y)\in L^p_\Pp$, that is, $y\in\tilde U$, which yields a contradiction. In this manner, $\tilde U=U$, and \ref{it:lcl-b} follows.

\smallskip
Conversely, assume that \ref{it:lcl-b} holds and fix $x\in U$. Let $\kappa\in L^p_\Pp$ be a Lipschitz modulus of $\varphi$ on a neighborhood of $x$. Then $\|v\|\leq\kappa(\xi)$ for every $\xi\in\Xi$ and every $v\in\partial\varphi_\xi(x)$, so $\dist(0,\partial\varphi_\xi(x))\leq\kappa(\xi)$ and \eqref{eq_X} follows.
\end{proof}

\begin{remark}[On bounded-variance and multiplicative-noise assumptions]
	\label{rem:variance}
	In smooth stochastic optimization, it is common to impose square-integrability of the stochastic gradient noise. For instance, one assumes that
	\[
	\sigma(x):=
	\left(
	\E\big[
	\|\nabla_x\varphi(\xi,x)-\nabla \Intf{\varphi}(x)\|^2
	\big]
	\right)^{1/2}
	<+\infty,
	\]
	or that a \emph{multiplicative-noise condition} of the form
	\begin{equation}\label{eq:mult-noise}
		\sigma(x)\leq \sigma(\hat x)+\sigma_L\|x-\hat x\|,
		\qquad x\in U,
	\end{equation}
	holds for some $\hat x\in U$ and some constant $\sigma_L>0$. It is straightforward to show that either condition implies
	\[
	\E\big[\|\nabla_x\varphi(\xi,x)\|^2\big]<+\infty
	\quad\text{for all }x\in U, 
	\] and thus \eqref{eq_X} holds for $p = 2$. Hence, under the assumptions of Corollary~\ref{cor:lowerC2_integrand_lip} with $p=2$, $\varphi$ is locally Lipschitz continuous around every point of $U$ with square-integrable Lipschitz modulus.
	
\end{remark}



\section{Applications in structured optimization} \label{sec:app}

In this section, we describe two structured applications that fit our framework of nonconvex normal integrands.

\subsection{Partial-label regression}\label{subsec:partial_label} 

In partial-label regression,  each observation is attached to a set of (real-valued) candidate labels that is known to contain the true one \cite{cheng2023partial,hullermeier2014learning}. A way of fitting a model to such data is to be optimistic with respect to the ambiguity, namely, the model chooses the most favorable among all candidate label assignments. We show that the resulting optimistic distributionally robust formulation reduces to the expectation of an upper-$C^2$ normal integrand. %

Let $m\in\N$. The data consist of a measurable set-valued mapping $S:\Xi\setto\R^m$ with nonempty closed values, representing the candidates, and a Carath\'eodory mapping $G:\Xi\times\R^d\to\R^m$, the prediction model. A typical instance is $G_\xi(\cdot)=h(\cdot;a(\xi))$, where $a(\xi)$ is a random feature vector and $h$ is a parametric model. For the squared residual, the optimistic loss at $x \in \R^d$ of a sample $\xi$ is
\begin{equation}\label{eq:pl_integrand}
	\varphi_\xi(x):=\min_{y\in S(\xi)}\frac12\|G_\xi(x)-y\|^2=\frac12\dist\big(G_\xi(x),S(\xi)\big)^2 .
\end{equation}
Let $\mathcal Q$ be the set of probability measures on $\Xi\times\R^m$, endowed with the product $\sigma$-algebra $\mathcal F\otimes\mathcal B(\R^m)$, consisting of all joint laws of observations and labels that reproduce the observed data and assign to each observation a (possibly randomized) label among its candidates. In other words, $Q \in \mathcal Q$ if
\begin{equation}\label{eq:pl_ambiguity}
	Q(D\times\R^m)=\Pp(D)\quad\text{for all } D\in\mathcal F
	\qquad\text{and}\qquad
	Q(\gph S)=1 .
\end{equation}Notice that $\gph S\in\mathcal F\otimes\mathcal B(\R^m)$ because $S$ is measurable with closed values \cite[Theorem~14.8]{rockafellarwets1998}.

For a linear map $A$, $\|A\|$ denotes the operator norm, and for $x\in\R^d$ we regard $\nabla^2G_\xi(x)$ as a symmetric bilinear map $\R^d\times\R^d\to\R^m$ with norm
\[
\|\nabla^2G_\xi(x)\|:=\sup\big\{\|\nabla^2G_\xi(x)[u,w]\|:\ \|u\|\leq1,\ \|w\|\leq1\big\}.
\]

\begin{proposition}\label{prop:partial_label}
	Let $S$, $G$ and $\varphi$ be as above. Fix $\bar x\in\R^d$ and $\delta>0$, set $V:=\B_\delta[\bar x]$, and assume that, for every $\xi\in\Xi$, the mapping $G_\xi$ is twice-continuously differentiable on an open set containing $V$, and that there exist measurable functions $K_1,K_2:\Xi\to[0,+\infty)$ such that
	\begin{equation}\label{eq:pl_bounds_G}
		\|\nabla G_\xi(x)\|\leq K_1(\xi),
		\qquad
		\|\nabla^2 G_\xi(x)\|\leq K_2(\xi)
		\qquad\text{for all } x\in V,\ \xi\in\Xi .
	\end{equation}
	Set $r_0(\xi):=\dist\big(G_\xi(\bar x),S(\xi)\big)$. Then the following assertions hold.
	\begin{enumerate}[label=(\alph*)]
		\item\label{it:pl-a} For every $x\in\R^d$,
		\begin{equation}\label{eq:pl_dro_identity}
			\inf_{Q\in\mathcal Q}\int_{\Xi\times\R^m}\frac12\|G_\xi(x)-y\|^2\,Q\big(d(\xi,y)\big)
			=\Intf{\varphi}(x),
		\end{equation}
		where both sides belong to $[0,+\infty]$ and the infimum is attained.
		
		\item\label{it:pl-b} The function $\varphi$ is a finite-valued normal integrand which is upper-$C^2$ around $\bar x$ in the sense of Definition~\ref{def:upperC2normal}, with Polish space $X:=\R^m$, compact-valued measurable mapping
		\begin{equation}\label{eq:pl_C}
			C(\xi):=S(\xi)\cap\B_{\tau(\xi)}\big[G_\xi(\bar x)\big],
			\quad \mbox{where} \quad
			\tau(\xi):=r_0(\xi)+2\delta K_1(\xi),
		\end{equation}
		Carath\'eodory mapping $g(\xi,y,x):=\frac12\|G_\xi(x)-y\|^2$, and curvature modulus
		\begin{equation}\label{eq:pl_rho}
			\varrho(\xi):=K_1(\xi)^2+K_2(\xi)\big(r_0(\xi)+3\delta K_1(\xi)\big).
		\end{equation}
		
		\item\label{it:pl-c} For every $\xi\in\Xi$, the function $\varphi_\xi$ is Lipschitz continuous on $V$ with modulus
		\begin{equation}\label{eq:pl_kappa}
			\kappa(\xi):=K_1(\xi)\big(r_0(\xi)+\delta K_1(\xi)\big).
		\end{equation}
		
		\item\label{it:pl-d} Let $p\in[1,\infty]$. If $K_1,r_0\in L^{2p}_\Pp$ and $K_2r_0,\,K_2K_1\in L^p_\Pp$, then $\varrho,\kappa\in L^p_\Pp$ and $\varphi(\cdot,x)\in L^p_\Pp$ for every $x\in V$.
	\end{enumerate}
\end{proposition}

\begin{proof}
	For every $w\in\R^m$, the function $\xi\mapsto\dist(w,S(\xi))$ is measurable since $S$ is measurable with closed values \cite[Theorem~14.3]{rockafellarwets1998}. In particular, $r_0$ is measurable. As $w\mapsto\dist(w,S(\xi))$ is $1$-Lipschitz, $(\xi,w)\mapsto\dist(w,S(\xi))$ is Carath\'eodory, and so is its composition with the Carath\'eodory mapping $G$. Consequently, $\varphi$ is a finite-valued Carath\'eodory integrand, and thus a normal integrand \cite[Example~14.29]{rockafellarwets1998}. 

    We first show \ref{it:pl-a}. Fix $x\in\R^d$ and set $f_x(\xi,y):=\frac12\|G_\xi(x)-y\|^2$, which is Carath\'eodory and hence $\mathcal F\otimes\mathcal B(\R^m)$-measurable. Let $Q\in\mathcal Q$. Since $Q(\gph S)=1$, we have $f_x(\xi,y)\geq\varphi_\xi(x)$ for $Q$-almost every $(\xi,y)$, and since the $\Xi$-marginal of $Q$ is $\Pp$, it follows that
	\[
	\int f_x\,dQ\geq\int\varphi_\xi(x)\,Q\big(d(\xi,y)\big)=\Intf{\varphi}(x).
	\]
	Conversely, the mapping $\xi\mapsto\proj_{S(\xi)}(G_\xi(x))$ is the argmin mapping of the normal integrand $(\xi,y)\mapsto f_x(\xi,y)+\iota_{S(\xi)}(y)$, which is level-bounded in $y$. Hence it is measurable with nonempty closed values \cite[Theorem~14.37]{rockafellarwets1998} and admits a measurable selection $\hat y_x(\xi) \in \proj_{S(\xi)}(G_\xi(x))$ \cite[Corollary~14.6]{rockafellarwets1998}. The image $Q_x$ of $\Pp$ under $\xi\mapsto(\xi,\hat y_x(\xi))$ belongs to $\mathcal Q$, and
	\[
	\int f_x\,dQ_x=\int_\Xi \frac12\|G_\xi(x)-\hat y_x(\xi)\|^2\,\Pp(d\xi)=\Intf{\varphi}(x),
	\]
	which proves \eqref{eq:pl_dro_identity} together with the attainment of the infimum.
	
	Next, to show \ref{it:pl-b}, we first check that $C$ in \eqref{eq:pl_C} is a measurable set-valued mapping with nonempty compact values. The mapping $\xi\mapsto\B_{\tau(\xi)}[G_\xi(\bar x)]$ is measurable with closed values, because it is a level-set mapping of a normal integrand \cite[Proposition~14.33]{rockafellarwets1998}, and thus so is the intersection in \eqref{eq:pl_C} \cite[Proposition~14.11]{rockafellarwets1998}. Its values are closed and bounded, hence compact. Now fix $\xi\in\Xi$ and $x\in V$. Since $V$ is convex, \eqref{eq:pl_bounds_G} and the mean value inequality give
	\begin{equation}\label{eq:pl_mvt}
		\|G_\xi(x)-G_\xi(\bar x)\|\leq K_1(\xi)\|x-\bar x\|\leq\delta K_1(\xi).
	\end{equation}
	Let $y\in\proj_{S(\xi)}(G_\xi(x))$. As $\dist(\cdot,S(\xi))$ is $1$-Lipschitz, \eqref{eq:pl_mvt} yields
	\[
	\|y-G_\xi(\bar x)\|
	\leq \dist\big(G_\xi(x),S(\xi)\big)+\|G_\xi(x)-G_\xi(\bar x)\|
	\leq r_0(\xi)+2\|G_\xi(x)-G_\xi(\bar x)\|
	\leq\tau(\xi),
	\]
	that is, $\proj_{S(\xi)}(G_\xi(x))\subseteq C(\xi)$. In particular, by taking  $x=\bar x$, $C(\xi)\neq\emptyset$. Since $C(\xi)\subseteq S(\xi)$, we conclude that
	\begin{equation}\label{eq:pl_localized_min}
		\varphi_\xi(x)=\min_{y\in C(\xi)}g(\xi,y,x)
		\qquad\text{for all } x\in V,
	\end{equation}
	which is \eqref{eq:smooth_value_representation}. Moreover, the mapping $g$ is Carath\'eodory. Indeed, for each $\xi\in\Xi$ and $y\in\R^m$, the function $x\mapsto g(\xi,y,x)$ is twice-continuously differentiable on an open set containing $V$, such that
	\[
	\nabla_xg(\xi,y,x)=\nabla G_\xi(x)^\top\big(G_\xi(x)-y\big)
	\] and \[\nabla_x^2g(\xi,y,x)[u,u]=\|\nabla G_\xi(x)u\|^2+\big\langle G_\xi(x)-y,\nabla^2G_\xi(x)[u,u]\big\rangle ,
	\]
	and $g$, $\nabla_xg$ and $\nabla_x^2g$ are continuous in $(y,x)$. Moreover, for $y\in C(\xi)$ and $x\in V$, \eqref{eq:pl_mvt} gives
	\[
	\|G_\xi(x)-y\|\leq\|G_\xi(x)-G_\xi(\bar x)\|+\|G_\xi(\bar x)-y\|\leq\delta K_1(\xi)+\tau(\xi)=r_0(\xi)+3\delta K_1(\xi).
	\]
	Combining these estimates with \eqref{eq:pl_bounds_G}, for $\|u\|\leq1$,  we obtain
	\[
	\|\nabla_x^2g(\xi,y,x)\|\leq K_1(\xi)^2+K_2(\xi)\big(r_0(\xi)+3\delta K_1(\xi)\big)=\varrho(\xi)
	\quad\text{for all } y\in C(\xi),\ x\in V,
	\]
	which yields \eqref{eq:hessian_bound_g_lower}. Hence all the requirements of Definition~\ref{def:upperC2normal} are fulfilled, and $\varphi$ is an upper-$C^2$ normal integrand around $\bar x$ with curvature modulus \eqref{eq:pl_rho}.
	
	To show \ref{it:pl-c}, set $d_\xi(x):=\dist(G_\xi(x),S(\xi))$. By \eqref{eq:pl_mvt}, $d_\xi(x)\leq r_0(\xi)+\delta K_1(\xi)$ and $|d_\xi(x)-d_\xi(x')|\leq K_1(\xi)\|x-x'\|$ for all $x,x'\in V$. Therefore
	\[
	|\varphi_\xi(x)-\varphi_\xi(x')|=\frac12\,|d_\xi(x)-d_\xi(x')|\,\big(d_\xi(x)+d_\xi(x')\big)\leq\kappa(\xi)\|x-x'\| .
	\]
	
	We conclude by proving \ref{it:pl-d}. Since $K_1\in L^{2p}_\Pp$, we have $K_1^2\in L^p_\Pp$, and by H\"older's inequality $\|K_1r_0\|_p\leq\|K_1\|_{2p}\|r_0\|_{2p}$. Hence $\varrho=K_1^2+K_2r_0+3\delta K_2K_1$ and $\kappa=K_1r_0+\delta K_1^2$ belong to $L^p_\Pp$. Finally, $0\leq\varphi_\xi(x)\leq\frac12\big(r_0(\xi)+\delta K_1(\xi)\big)^2$ for all $x\in V$, and the right-hand side belongs to $L^p_\Pp$.
\end{proof}

\begin{remark}[Optimism versus pessimism]\label{rem:pl_pessimism}
	If $S$ has compact values, replacing the infimum in \eqref{eq:pl_dro_identity} by a supremum leads, by the same argument, to the integrand $\max_{y\in S(\xi)}\frac12\|G_\xi(x)-y\|^2$. This integrand is lower-$C^2$ by Definition~\ref{def:upperC2normal}, so pessimism and optimism over candidate labels correspond, respectively, to the lower- and upper-$C^2$ classes. When $S(\xi)$ is convex, as for interval-valued labels, the function $\frac12\dist(\cdot,S(\xi))^2$ is convex with a $1$-Lipschitz gradient. Then $\varphi_\xi$ is convex for linear models, and smooth whenever $G_\xi$ is.
\end{remark}

\begin{remark}[Descent property on active candidate labels]\label{rem:pl_descent}
By Corollary~\ref{cor:eqC2upper}\ref{it:eqC2u-b}, Proposition~\ref{prop:partial_label} yields the descent inequality
\[
\varphi_\xi(x')\leq\varphi_\xi(x)+\langle v,x'-x\rangle+\frac{\varrho(\xi)}{2}\|x'-x\|^2
\qquad\text{for all } x,x'\in V \text{ and all } v\in\partial\varphi_\xi(x),
\]
where, by the localized representation \eqref{eq:pl_localized_min} and  \cite[Theorem 10.31 ]{rockafellarwets1998} applied to $-\varphi_\xi$,
\begin{equation}\label{eq:pl_clarke}
	\partial\varphi_\xi(x)=\conv\Big\{\nabla G_\xi(x)^\top\big(G_\xi(x)-y\big):\ y\in\proj_{S(\xi)}\big(G_\xi(x)\big)\Big\}.
\end{equation}
In particular, the descent inequality holds for every choice of active candidate label. This is related to identification methods for partial-label learning \cite{cheng2023partial,zhang2017reshaped}, in which the updates first identify a candidate (or model) and then perform the optimization step.
\end{remark}

\begin{example}[Amplitude-based phase retrieval]\label{ex:amplitude_PR}
	Let $a:\Xi\to\R^d$ and $b:\Xi\to[0,+\infty)$ be measurable. The amplitude-based loss underlying the reshaped Wirtinger flow \cite{zhang2017reshaped} and the truncated amplitude flow \cite{wang2018solving} is
	\begin{equation}\label{eq:amplitude_loss}
		\varphi_\xi(x):=\frac12\big(|\langle a(\xi),x\rangle|-b(\xi)\big)^2 .
	\end{equation}
	This is the particular case of \eqref{eq:pl_integrand} with $m=1$, the linear model $G_\xi(x)=\langle a(\xi),x\rangle$, and the two-point candidate set $S(\xi):=\{-b(\xi),b(\xi)\}$. Indeed, $S$ is measurable with nonempty compact values, and since $b(\xi)\geq0$,
	\[
	\dist\big(u,\{-b(\xi),b(\xi)\}\big)=\min\big\{|u-b(\xi)|,|u+b(\xi)|\big\}=\big||u|-b(\xi)\big|
	\qquad\text{for all } u\in\R .
	\]By Proposition~\ref{prop:partial_label}, $\varphi$ is an upper-$C^2$ normal integrand around every point of $\R^d$, with $r_0(\xi)=\big||\langle a(\xi),\bar x\rangle|-b(\xi)\big|$, $\kappa(\xi)=\|a(\xi)\|\big(r_0(\xi)+\delta\|a(\xi)\|\big)$ and $\varrho(\xi)=\|a(\xi)\|^2$. By \eqref{eq:pl_clarke}, the Clarke subdifferential is
	\[
	\partial\varphi_\xi(x)=
	\begin{cases}
		\big\{\big(\langle a,x\rangle-b\,\mathrm{sgn}\langle a,x\rangle\big)a\big\}, & \langle a,x\rangle\neq0,\\[1mm]
		[-b,b]\,a, & \langle a,x\rangle=0,
	\end{cases}
	\]
	where the dependence on $\xi$ is omitted. Consequently, the descent inequality
	\[
	\varphi_\xi(x')\leq\varphi_\xi(x)+\langle v,x'-x\rangle+\frac{\|a(\xi)\|^2}{2}\|x'-x\|^2
	\]
	holds for all $x,x'\in\R^d$ and every $v\in\partial\varphi_\xi(x)$. In particular, it holds for the update direction of the reshaped Wirtinger flow \cite[Algorithm 1]{zhang2017reshaped}. The update rule in \cite[Algorithm 1]{wang2018solving} of the truncated amplitude flow acts on these subgradients, by discarding some samples at each iteration.  
\end{example}

 \subsection{Conditional value-at-risk}

The conditional value-at-risk (CVaR), introduced in \cite{rockafellar2000optimization}
(see also \cite{rockafellar2002conditional, acerbi2002coherence}), is a risk measure
that quantifies the expected loss in the tail of a distribution. Let $\alpha\in(0,1)$
be a confidence level and let $L:\Xi\times\R^d\to\R$ be a finite-valued normal
integrand, where $L_\xi(x)$ represents the loss incurred by the decision $x\in\R^d$
under the scenario $\xi\in\Xi$, with $L(\cdot,x)\in L^1_\Pp$. Roughly speaking, the
CVaR of $L_\xi(x)$ is the average of the worst $(1-\alpha)$-fraction of the losses.
When the distribution of $L_\xi(x)$ is continuous, it coincides with the conditional
expectation
\[
\mathrm{CVaR}_\alpha\bigl(L_\xi(x)\bigr)
  = \E\bigl[L_\xi(x)\,\big|\,L_\xi(x)\ge \mathrm{VaR}_\alpha\bigl(L_\xi(x)\bigr)\bigr],
\]
where $\mathrm{VaR}_\alpha$ denotes the value-at-risk, that is, the $\alpha$-quantile
of the loss. Unlike the VaR, the CVaR is a coherent risk measure in the sense of
\cite{artzner1999coherent}, and it interpolates between the expected loss
($\alpha\to 0$) and the worst-case loss ($\alpha\to 1$).

The CVaR has found applications in a wide range of fields. In finance, where it
originated, it is used for portfolio selection and hedging
\cite{rockafellar2000optimization}, and it has replaced the VaR as the reference
measure of market risk in the Basel regulatory framework; see
\cite{mcneil2015quantitative} for a comprehensive account. In stochastic programming
it is the prototypical coherent risk measure for risk-averse formulations
\cite{MR4362585}, and CVaR constraints provide a convex conservative
approximation of chance constraints, the tightest one within a natural class
\cite{nemirovski2006convex}. More recently, the CVaR has been adopted in machine
learning as a training criterion that focuses on the hardest examples or
subpopulations rather than on the average loss, owing to its dual representation as
a worst-case expectation over distributions with bounded likelihood ratio
\cite{duchi2021learning}, and in risk-sensitive reinforcement learning and control
\cite{chow2017risk}. In most of these applications the CVaR enters either as an
objective or as a constraint, so that its tractability for optimization is
essential.

A key feature of the CVaR is that it admits a variational formulation
\cite{rockafellar2000optimization}, which makes it amenable to optimization. It can
be equivalently described by
\[
\mathrm{CVaR}_\alpha\bigl(L_\xi(x)\bigr)=\min_{t\in\R}\,\Intf{\psi}(x,t),
\]
with
\begin{equation}\label{eq:psiCvAR}
\psi_\xi(x,t):=t+\frac{1}{1-\alpha}\bigl(L_\xi(x)-t\bigr)_+ ,\qquad (x,t)\in\R^{d+1},
\end{equation}
where $(s)_+:=\max\{s,0\}$. Moreover, the value-at-risk $\mathrm{VaR}_\alpha\bigl(L_\xi(x)\bigr)$
is a minimizer of the inner problem in $t$ (the left endpoint of the set of minimizers).
Consequently, minimizing the CVaR over $x$ reduces to the joint stochastic problem
\begin{equation}\label{eq:CvAR}
\min_{(x,t)\in\R^{d+1}} \Intf{\psi}(x,t).
\end{equation}
 The next proposition shows that the integrand $\psi$ in \eqref{eq:CvAR} inherits
the lower-$C^2$ structure and the Lipschitz continuity of the loss $L$. It is a
direct consequence of Proposition~\ref{prop:composite_monotone}.

\begin{proposition}\label{prop:cvar}
Let $p\in[1,\infty]$, let $L:\Xi\times\R^d\to\R$ be a finite-valued normal
integrand, and let $\bar x\in\R^d$ be such that $L$ is lower-$C^2$ around $\bar x$
with $p$-integrable curvature modulus. Then the integrand $\psi$ defined in
\eqref{eq:psiCvAR} is a finite-valued normal integrand on $\Xi\times\R^{d+1}$ which
is lower-$C^2$ around $(\bar x,\bar t)$ with $p$-integrable curvature modulus, for
every $\bar t\in\R$. If, in addition, $L$ is locally Lipschitz around $\bar x$ with
$p$-integrable modulus, then $\psi$ is locally Lipschitz around $(\bar x,\bar t)$
with $p$-integrable modulus, for every $\bar t\in\R$.
\end{proposition}

\section{Concluding comments}\label{sec:conclusion}

The preceding results show that lower-$C^2$ normal integrands admit a simple local representation as upper envelopes of random quadratic-affine functions with integrable curvature, and, by a change of sign, that upper-$C^2$ normal integrands are the corresponding lower envelopes. This representation is suitable for stochastic programming, since it is stable under compact-recourse minimization and provides a direct route from samplewise second-order estimates to expected-value estimates. A natural next step is to combine this structure with algorithmic frameworks for nonsmooth stochastic optimization, where the integrable curvature modulus and square-integrable Lipschitz moduli can be used to control model errors and sample-average approximations.

\section*{Acknowledgements}

Research of the authors was supported by Centro de Modelamiento Matem\'atico (CMM), ACE210010 and FB210005, BASAL funds for centers of excellence. The second author was also supported by ANID-Chile grants Fondecyt Regular 1240120, Fondecyt Regular 1240335 and Fondecyt Regular 1261728. The fourth author was  partially supported by grant PID2025-168246NB-I00 funded by
ERDF/EU and by MICIU/AEI/10.13039/501100011033. The last two authors were partially supported by grant CIGE/2026/80 by Conselleria d'Educació, Cultura i Universitats, Generalitat Valenciana.
\bibliographystyle{siam}
\bibliography{references}

\end{document}